\documentclass[12pt]{amsart}

\usepackage{amssymb,amsfonts,amsmath}
\usepackage[all,arc]{xy}
\usepackage{enumerate}
\usepackage{mathrsfs}
\usepackage{fullpage}
\usepackage{xspace}
\usepackage[margin=1.0in]{geometry}
\usepackage{tcolorbox}
\usepackage{tikz-cd}
\usepackage{color}
\usepackage{aliascnt}
\usepackage[foot]{amsaddr}
\usepackage{hyperref}
\usepackage{mathtools}
\usepackage{pdflscape}
\usepackage{booktabs}
\usepackage{makecell}

\newtheorem{thm}{Theorem}[section]

\newaliascnt{theo}{thm}

\aliascntresetthe{theo}

\newaliascnt{cor}{thm}
\newtheorem{cor}[cor]{Corollary}
\aliascntresetthe{cor}

\newaliascnt{prop}{thm}
\newtheorem{prop}[prop]{Proposition}
\aliascntresetthe{prop}

\newaliascnt{lem}{thm}
\newtheorem{lem}[lem]{Lemma}
\aliascntresetthe{lem}

\newaliascnt{conj}{thm}

\aliascntresetthe{conj}

\newaliascnt{que}{thm}

\aliascntresetthe{que}

\newaliascnt{ass}{thm}

\aliascntresetthe{ass}

\newaliascnt{defn}{thm}
\newtheorem{defn}[defn]{Definition}
\aliascntresetthe{defn}

\theoremstyle{remark}
\newaliascnt{rem}{thm}
\newtheorem{rem}[rem]{Remark}
\aliascntresetthe{rem}

\theoremstyle{definition}

\newtheorem{set}[thm]{Setting}

\newtheorem{mainthm}{Theorem}

\newtheorem{maincorollary}{Corollary}

\newtheorem{mainexample}{Example}

\newcommand{\C}{\mathbb{C}\xspace}

\newcommand{\Z}{\mathbb{Z}\xspace}
\newcommand{\Q}{\mathbb{Q}\xspace}
\newcommand{\G}{\mathbb{G}\xspace}
\newcommand{\A}{\mathbb{A}\xspace}
\newcommand{\F}{\mathbb{F}\xspace}

\newcommand{\ch}{\textup{CH}}

\newcommand{\et}{\text{ét}}
\newcommand{\into}{\hookrightarrow}
\newcommand{\Br}{\text{Br}}
\DeclareMathOperator{\Spec}{Spec}

\DeclareMathOperator{\dv}{div}
\DeclareMathOperator{\alb}{alb}

\DeclareMathOperator{\Cor}{Cor}

\DeclareMathOperator{\Hom}{Hom}
\DeclareMathOperator{\Pic}{Pic}
\DeclareMathOperator{\Gal}{Gal}

\DeclareMathOperator{\Alb}{Alb}

\makeatletter
\let\c@equation\c@theo
\makeatother
\numberwithin{equation}{section}

\title{Zero-cycles and Brauer groups on certain Enriques surfaces}
\author{Thomas Jaklitsch}
\begin{document}
	\maketitle
	\begin{abstract}
        In this paper, we study the Chow group of zero-cycles $\ch_0(S)$ for certain Enriques quotients $S$ of the Kummer surface associated to the product of two geometrically non-isogenous elliptic curves $E_1$ and $E_2$. We show that over fields of characteristic different from $2$ or $3$, the degree-zero subgroup $A_0(S) \subset \ch_0(S)$ is 16-torsion. Over p-adic fields, we compute the transcendental Brauer group $\Br(S)/\Br_1(S)$. Using the Brauer-Manin pairing, we construct examples of Enriques surfaces for which $A_0(S)$ attains nontrivial 2-torsion elements arising from push-forward from the abelian surface $E_1\times E_2$. 
    \end{abstract}
	\section{Introduction}
    Let $Y$ be a smooth, projective, geometrically integral variety over a field $k$. The Chow group of zero-cycles admits a degree map
    \[
        \deg: \ch_0(Y) \to \Z
    \]
    which is defined on the class of a closed point $[y]$ by $\deg([y]) := [k(y):k]$. We denote the kernel of the degree map by $A_0(Y)$. If $Y(k)\neq \emptyset$, then there exists a homomorphism known as the Albanese map
    \[
        \alb_Y: A_0(Y) \to \Alb_Y(k)
    \]
    where $\Alb_Y$ denotes the Albanese variety of $Y$. Let us denote the kernel by $T(Y)$. In the case of curves, $\alb_Y$ is the Abel-Jacobi map, so it is an isomorphism. However, it is well known that for higher dimensions the Albanese kernel need not be trivial and can be quite large in general. For example, Mumford showed \cite{Mumford} that if $Y$ is a surface over $\C$ with geometric genus $p_g(Y) > 0$, then the kernel of the Albanese map is infinite dimensional, and it is not even representable. However, it is expected that this behavior should only hold in the case $p_g(Y) >0$. More precisely, Bloch conjectured \cite[Conjecture 0.4]{Bloch-BlochsConj} that if $Y$ is a smooth projective surface of geometric genus $0$, then the Albanese kernel $T(Y) = 0$. It is a theorem of Bloch, Kas, and Lieberman \cite{BlochKasLieberman} that the conjecture is true for surfaces not of general type. That is, they prove that if $Y$ is a smooth projective surface over $\C$ such that $p_g(Y) = 0$ and the Kodaira dimension $\kappa(Y) < 2$, then $T(Y) = 0$.\footnote{Bloch's conjecture has also been verified for many surfaces of general type (see for example \cite{Voisin} and \cite{PedriniWeibel}), and OpenAI has announced a proof of the full conjecture.} This theorem implies that for a smooth projective surface $X$ over a field of characteristic $0$, $T(Y)$ is torsion. A natural question then, is what can this group look like over fields of arithmetic interest and can this group be nontrivial. We note that this question has been studied for some classes of surfaces with $p_g = 0$ and not of general type. See \cite{Gazaki} for the case of bielliptic surfaces and \cite{C-RamachandranHogadi} for Severi-Brauer flag varieties (where the case of higher dimensions is also considered).
    
    In this paper, we study the Chow group of zero-cycles for a specific class of Enriques surfaces which we will refer to as Enriques surfaces of Lieberman type. An Enriques surface $S$ is a smooth projective surface with both geometric genus and irregularity equal to $0$ and nontrivial canonical divisor $K_S$ which satisfies $2K_S \sim 0$. Over fields of characteristic not equal to $2$, Enriques surfaces arise as quotients of K3 surfaces by fixed-point-free involutions. Since the irregularity of an Enriques surface is $0$, $\Alb_S$ is trivial, and therefore the degree $0$ subgroup coincides with the Albanese kernel. It then follows by the above that over $\C$, the degree map $\deg: \ch_0(S) \to \Z$ is an isomorphism. Our goal is to study the torsion subgroup $A_0(S)$ for non algebraically closed fields. First, we give an upper bound on the torsion of this group over an arbitrary field of characteristic not equal to $2$ or $3$. Second, we construct examples of Enriques surfaces over p-adic fields for which $A_0(S)$ possesses interesting nontrivial classes arising from pushforward from an abelian surface. 

    \subsection{Main Results}

    Let $A:= E_1\times E_2$ denote the product of two elliptic curves over a field $k$ which are geometrically non-isogenous and have rational 2-torsion. We consider the associated Kummer K3 surface $X := \text{Kum}(A)$. Given 2-torsion points $P \in E_1[2]$ and $Q \in E_2[2]$ which are not equal to the identity, we define the involution 
    \[
        \sigma_{P,Q}: A \to A
    \]
    by $\sigma_{P,Q}(x,y) = (-x+P, y+Q)$. This induces an involution $\tilde{\sigma}_{P,Q}$ on $X$ and an Enriques surface of Lieberman type is a quotient of $X$ by an involution of this form (see Section \ref{sec: EnSurfPre}).  We denote the rational map $f: A \dashrightarrow S$. Our first result is determining an upper-bound on the torsion of these groups over a general field.
    
    \begin{mainthm}[Theorem \ref{thm: ChUpperBound}]
        Let $k$ be a field of characteristic different than 2 or 3 and let $S$ be an Enriques surface of Lieberman type over $k$. Then the kernel of the degree map $A_0(S)$ is $16$-torsion.
    \end{mainthm}

    \subsection*{Results over p-adic fields}
    We now specialize to the case when $k$ is a finite extension of $\Q_p$ for $p > 2$. We further assume that there exists a number field $F \subset k$ such that $E_1$ and $E_2$ are defined over $F$. More precisely, we assume that there exist elliptic curves $\tilde{E}_1, \tilde{E}_2$ defined over $F$ which are geometrically non-isogenous and have rational 2-torsion such that $E_1 \cong \tilde{E}_1\times_F k$ and $E_2 \cong \tilde{E}_2\times_F k$. Our main goal is to compute the group $A_0(S)$ for the Enriques surface $S$ under these assumptions. In the case of good reduction, it follows easily by the geometric class field theory of Kato and Saito \cite{KatoSaito} that $A_0(S) \cong \Z/2\Z$ (see Section \ref{sec: GoodRed}). Therefore, we study the case when both elliptic curves have split multiplicative reduction. In order to do this, the main tool we make use of is the Brauer-Manin pairing.

    Following work of Lichtenbaum \cite{Lichtenbaum} on curves, Colliot-Thélène \cite[pg 53]{Colliot-Thelene} defined for a smooth projective variety $Y$ over a p-adic field $k$, a functorial pairing 
    \[
        \ch_0(Y) \times \Br(Y) \to \Q/\Z.
    \]
    Moreover, for $n > 0$ this induces a pairing 
    \[
        A_0(Y)/n \times \Br(Y)[n]/\Br(k)[n] \to \Q/\Z.
    \]
    Our strategy to show that $A_0(S) \neq 0$ is to show that the homomorphism
    \[
        A_0(S)/2 \to \Hom(\Br(S)[2]/\Br(k)[2], \Q/\Z)
    \]
    induced by the Brauer-Manin pairing is nontrivial. In fact, we use this to construct examples where the pushforward $f_*T(A)$ is nontrivial in $A_0(S)$.

    A key step in this direction is to compute the transcendental Brauer group 
    of $S$, $\Br(S)/\Br_1(S)$. For an Enriques surface of Lieberman type, we compute the transcendental Brauer group \footnote{We believe that this computation may be known to experts.} as well as the image of the composition
    \[
        \Br(S)/\Br_1(S) \to \Br(X)[2]/\Br_1(X)[2] \to \Br(A)[2]/\Br_1(A)[2]
    \]
    (See Theorem \ref{thm: CalcBrS}). The exact structure of $\Br(S)/\Br_1(S)$ only depends upon whether certain constants arising from the coefficients of the elliptic curves are squares in the p-adic field $k$. This result relies heavily on the work of Skorobogatov and Zarhin \cite{SkorobogatovZarhin} who study the Brauer group of $X$ and its relationship to the Brauer group of $A$. 
    We then use our computation of the transcendental Brauer group of the Enriques surface, p-adic uniformization of elliptic curves, and related work of Yamazaki \cite{Yamazaki} on the Brauer-Manin pairing of the abelian surface to prove the following theorem.
    
    \begin{mainthm}[Theorem \ref{thm: MainResult}] \label{thm: MainResultIntro}
        There exists an Enriques quotient of Lieberman type $S$ of the Kummer surface $X$ such that
        \[
            f_*T(A)/2 \to \Hom(\Br(S)[2]/\Br(k)[2], \Q/\Z)
        \]
        is nontrivial if and only if the map 
        \[
            \Br(X)[2]/\Br_1(X)[2] \to \Br(A)[2]/\Br_1(A)[2]
        \]
        is an isomorphism.
    \end{mainthm}

    Moreover, the Enriques surfaces that have this property can be explicitly identified. This theorem gives rise to the following corollary.
    
    \begin{maincorollary}[Corollary \ref{cor: FiniteExt}]
        There exists a finite extension $L/k$ and an Enriques quotient $S$ of $X$ such that $A_0(S_L)$ is nontrivial. 
    \end{maincorollary}
    Moreover, the finite extension $L/k$ and the Enriques surface $S$ can be made explicit in the above corollary. By the work of Skorobogatov and Zarhin \cite{SkorobogatovZarhin}, the condition in Theorem \ref{thm: MainResultIntro} is easy to computationally verify in many cases. Thus, it can be used to find explicit examples of Enriques surfaces whose Chow group has nontrivial degree-zero subgroup over p-adic fields.
    
    \begin{mainexample}[Section \ref{sec: Example}]
        Consider the elliptic curves
        \[
            E_1: y^2 + xy + y = x^3 + x^2 - 565x - 5353 \text{ (LMFDB 390.f3)} 
        \]
        \[
            E_2: y^2 + xy= x^3 -474x +2619 \text{ (LMFDB 429.b3)}
        \]
        over $\Q_{13}$. These curves are geometrically non-isogenous, have rational 2-torsion, and have split multiplicative reduction mod 13. Also, $\Br(X)[2]/\Br_1(X)[2] \to \Br(A)[2]/\Br_1(A)[2]$ is an isomorphism. Therefore, there exists an Enriques quotient $S$ of Lieberman type with the property that $A_0(S) \neq 0$. We can also identify which involutions of $A$ give rise to the Enriques surfaces $S$ with this property.
    \end{mainexample} 

    Theorem \ref{thm: MainResultIntro} can also be used to show that certain Enriques surfaces of the kind appearing in the theorem have a transcendental Brauer-Manin obstruction to weak approximation for rational points. To our knowledge, this is the second known example of a transcendental Brauer-Manin obstruction to weak approximation for an Enriques surface. The first is found in \cite{CreutzViray}. We also note that it was known that weak approximation need not hold for Enriques surfaces by \cite{HarariSkorobogatov} using a different method. This application is included in section \ref{sec: Example}.
    
    \begin{rem}
    The main results of this paper give explicit information about the Chow group of specific Enriques surfaces over local fields, which we hope to use in the future to study local-to-global principles for zero-cycles on these surfaces. In particular, we would like to study Colliot-Thélène's Conjecture (E) \cite{Colliot-TheleneSansuc}, \cite{KatoSaitoConjE}, \cite{Colliot-Thelene} for Enriques surfaces. This conjecture says that for a smooth, projective, geometrically integral variety $Y$ over a number field, the complex
    \[
        \varprojlim_n \ch_0(Y)/n \to \varprojlim_n \ch_{0, \A}(Y)/n \to \Hom(\Br(Y), \Q/\Z)
    \]
    is exact. This form of the conjecture is due to \cite{vanHamel} and \cite{Wittenberg}. A key input to construct the nontrivial zero-cycles coming from the abelian surface in this paper is the injectivity of the map 
    \[
        \Br(S)/\Br_1(S) \to \Br(X)/\Br_1(X).
    \]
    We note that Enriques surfaces with this property have given counterexamples where the Brauer group cannot explain the failure of Weak Approximation for rational points \cite{HarariSkorobogatov}. This makes the study of the zero-cycles counterpart very intriguing. 
    \end{rem}
    
    \subsection{Notation and Conventions}

    Varieties will always be assumed to be geometrically integral. In particular, a surface will always mean a geometrically integral variety of dimension 2. For an irreducible scheme $Y$, let $Y^{(1)}$ denote the set of points of codimension $1$. That is the set $\{y \in Y| \overline{\{y\}} \text{ has codimension $1$ in $Y$}\}$. For an abelian group $B$ we denote the n-torsion subgroup by $B[n]$. For an abelian variety $A$, we denote the $n$-torsion subscheme by $A[n]$. If $Y$ is a variety over $k$ and $L/k$ is a field extension, we denote the base change to $L$ by $Y_L := Y\times_{\Spec k} \Spec L$. In the case that $L=\overline{k}$ is a fixed algebraic closure, we denote the base change to $\overline{k}$ by $\overline{Y} := Y_{\overline{k}}$ . 

    \subsection*{Acknowledgements}
    I would like to thank my advisor, Professor Evangelia Gazaki, for introducing me to this problem and for her constant support and guidance on this project. I would also like to thank Professor Olivier Wittenberg for his time and for explaining to me the application appearing in Section 5. 
    
	\section{Background on Kummer and Enriques Surfaces and Preliminary Results} \label{sec: Background}
	
	\subsection{Chow Groups}
	Here we recall some important facts we will need about the Chow group of zero-cycles of a smooth projective variety. For more information see \cite{Fulton}.
	
	Let $k$ be a field and $Y$ a smooth projective variety over $k$ such that $Y(k)\ne \emptyset$. Recall that we have a filtration 
    \[
        T(Y) \subset A_0(Y) \subset \ch_0(Y)
    \]
	where $A_0(Y)$ is the kernel of the degree map and $T(Y)$ is the kernel of the Albanese homomorphism $\alb_Y$. 
	
    If $f: Y \to Z$ is a proper morphism, then there is a functorial push-forward map 
	\[
		f_*: \ch_0(Y) \to \ch_0(Z).
	\]
	Moreover, $f_*$ restricts to a map between degree zero subgroups and Albanese kernels. That is, $f_*(A_0(Y)) \subset A_0(Z)$ and $f_*(T(Y)) \subset T(Z)$  (see \cite{SpiessSzamuely} pg 5). 
	If $f: Y \to Z$ is flat, then there is a functorial pullback
	\[
		f^*: \ch^i (Z) \to \ch^i(Y).
	\]
	In particular, if $f: Y\to Z$ is a finite flat map between smooth projective varieties, then there is a pullback 
	\[
		f^*: \ch_0(Z) \to \ch_0(Y).
	\]
	In this case $f^*$ restricts to a map between degree zero subgroups and Albanese kernels. That is, $f^*(A_0(Z)) \subset A_0(Y)$ and $f^*(T(Z)) \subset T(Y)$ (see \cite{SpiessSzamuely} pg 5, \cite{GazakiRathore} pg 25). Another important fact that we will need is that a birational morphism of surfaces induces an isomorphism on $A_0$:
	
	\begin{prop} \label{prop:ChBirInv}\cite[Proposition 6.3]{Colliot-TheleneCoray}
	    Let $Y \to Z$ be a birational morphism of smooth projective surfaces over $k$. Then 
		\[
			f_* : A_0(Y) \to A_0(Z)
		\]
		is an isomorphism. 
	\end{prop}

    \subsection{Kummer and Enriques Surfaces}
    In this section, we state some information about the relevant surfaces that appear in this paper. Throughout, $k$ is a field with characteristic different from 2 or 3. First, let us recall the classical birational invariants of surfaces which will be important in what follows.  

    Let $Y$ be a smooth projective surface over $k$ and let $K_Y$ denote the canonical divisor. 
    \begin{defn}
        The geometric genus of $Y$ is $p_g(Y) := h^2(Y, \mathcal{O}_Y)$. The irregularity of $Y$ is $q(Y) := h^1(Y, \mathcal{O}_Y)$. The Kodaira dimension $\kappa(Y)$ is defined to be 
        \[
            \kappa(Y) := \inf\left\{n \in \Z : \frac{h^0(Y,mK_Y) }{m^n}\text{ is bounded}\right\}
        \]
        where we take $\kappa(Y) = -\infty$ if $h^0(Y,mK_Y) = 0$ for all $m$.
    \end{defn}

    Now let us introduce the surfaces we will be interested in. 

    \begin{defn}
        Let $A$ be an abelian surface over $k$. The Kummer surface $\text{Kum}(A)$ associated to $A$ is defined as follows. Let $\iota: A \to A$ be the involution $\iota(x) = -x$. The quotient $A/\iota$ is singular at the 2-torsion points of $A$ (not necessarily defined over $k$) which are the fixed points of the action. Then $\text{Kum}(A)$ is defined as the minimal resolution of $A/\iota$ which by \cite[Section 2.4]{Kollar} is obtained by blowing up the singular closed points of $A/\iota$. 
    \end{defn}

    The Kummer surface $X:= \text{Kum}(A)$ is a K3 surface over $k$ \cite[Example 1.3 (iii)]{Huybrechts}. In particular, it has birational invariants $p_g(X) = 1$, $q(X) = 0$, and $\kappa(X) = 0$ as $K_X \sim 0$. Another fact that will be relevant later is that the function field $k(X) = k(A)^\iota$, the fixed points of the function field of $A$ by the action of $\iota$.

    \begin{defn}
        An Enriques surface $S$ over $k$ is a smooth projective geometrically integral surface such that $p_g(S) = q(S) = 0$ and $2K_S \sim 0$, but $K_S\not\sim 0$. 
    \end{defn}
    Note that since $2K_S \sim 0$, Enriques surfaces have Kodaira dimension $\kappa(S) = 0$. 

    \begin{prop}\label{prop: K3QuotGivesEn}
        Let $X$ be a K3 surface, and $\sigma: X \to X$ a fixed point free involution on $X$. Then $S := X/\sigma$ is an Enriques surface. 
    \end{prop}
    \begin{proof}
        See the argument of \cite[Proposition 1.3.3]{CossecDolgachevLiedtke}.
    \end{proof}

    The last class of surfaces that will be relevant to us are bielliptic surfaces.

    \begin{defn}
        A bielliptic surface is the quotient of the product of two elliptic curves $E_1\times E_2$ by a finite group $G$ which acts on $E_2$ by translations and on $E_1$ such that $E_1/G \cong \mathbb{P}^1$. 
    \end{defn}

    Bielliptic surfaces $W$ have birational invariants $\kappa(W) = 0$, $q(W) = 1$, and $p_g(W) = 0$ \cite[Theorem VI.13]{BeauvilleComplexSurfaces}. 
	
	\subsection{Enriques Surfaces of Lieberman Type}\label{sec: EnSurfPre}
	
	Let us recall the following notation from the introduction. Let $k$ be a field of charactertistic not equal to $2$ or $3$ and let $E_1$ and $E_2$ be elliptic curves with rational 2-torsion. We let $A := E_1\times E_2$ be the abelian surface given by the product of the two elliptic curves and $X := \text{Kum}(A)$ denote the Kummer surface associated to $A$ as defined above. Fix 2-torsion points $O_{E_1} \neq P \in E_1[2]$ and $O_{E_2}\neq Q \in E_2[2]$.  Then we define an involution
	\[
	\sigma_{P,Q}: E_1\times E_2 \to E_1\times E_2
	\]
    of the abelian surface by
	\[
	\sigma_{P,Q}(X,Y) := (-X + P, Y + Q).
	\]
    
    \begin{lem}\label{lem: KumInv}
        The involution $\sigma_{P,Q}: A \to A$ induces a free involution $\tilde{\sigma}_{P,Q}$ on the Kummer surface $X$. 
    \end{lem}
    \begin{proof}
        Recall that we denote the negation involution on $A$ by $\iota$. The composition of $\sigma_{P,Q}$ and $\iota$ gives 
        \[
            \iota \sigma_{P,Q}(x,y) = \iota(-x+P, y+Q) = (x+P, -y+Q)
        \]
        since $P=-P$ and $Q = -Q$ as they are 2-torsion points. Also,
        \[
            \sigma_{P,Q}\iota(x,y) = \sigma_{P,Q}(-x,-y) = (x+P, -y+Q).
        \]
        Therefore, the two actions commute, so $\sigma_{P,Q}$ induces an action on the quotient $A/\iota$. Since $X$ is defined to be the blowup of $A/\iota$ at the 16 singular points, we have the diagram
        \[
            \begin{tikzcd}
                X \arrow{d}{\pi} & X \arrow{d}{\pi} \\
                A/\iota \arrow{r}{\sigma_{P,Q}} & A/\iota.
            \end{tikzcd}
        \]
        where $\pi$ is the morphism from the blow-up. It is easy to see from the definition of $\sigma_{P,Q}$ that $A[2]$ is invariant under the action, so it takes singular points to singular points on $A/\iota$. Denote the set of singular points by $Z$. Then $(\sigma_{P,Q}\circ \pi)^{-1}(Z)$ is the 16 lines on $X$ obtained from blowing up, so by the universal property of blowing up, $\sigma_{P,Q}$ induces an involution $\tilde{\sigma}_{P,Q}$ on $X$. It remains to check that this action is fixed-point free. By the commutativity of the diagram 
        \[
            \begin{tikzcd}
                X \arrow{d}{\pi} \arrow{r}{\tilde{\sigma}_{P,Q}} & X \arrow{d}{\pi} \\
                A/\iota \arrow{r}{\sigma_{P,Q}} & A/\iota,
            \end{tikzcd}
        \]
        it suffices to check that $\sigma_{P,Q}$ acts freely on $A/\iota$. But this amounts to verifying that $\sigma_{P,Q}(x,y)\neq (x,y)$ or $(-x,-y)$. Since $P,Q$ are nontrivial by assumption, $-x+ P \neq -x $ and $y \neq y+Q$, so $\sigma_{P,Q}(x,y) = (-x+P, y+Q) \neq (x,y)$ or $(-x,-y)$. 
    \end{proof}

    \begin{defn}
        We define the surface $S_{P,Q} := X/\tilde{\sigma}_{P,Q}$. By Lemma \ref{lem: KumInv} and Proposition \ref{prop: K3QuotGivesEn} this is an Enriques surface. We refer to Enriques surfaces of this form by Enriques surfaces of Lieberman type.
    \end{defn}
    We will occasionally omit the $P,Q$ from the notation of these Enriques surfaces when it will not lead to any confusion.
    
    \begin{rem}
        In \cite{Ohashi} Section 4, Ohashi proves that there are 15 Enriques quotients of a generic Kummer surface associated to a product of elliptic curves over $\C$. There are 9 that come from involutions of Lieberman type and 6 more that come from involutions of Kondo-Mukai type. 
    \end{rem}
	
	\subsection{Upper Bounds on Chow Groups of Enriques Surfaces}
	
	In this section, we show that the degree-zero subgroup of the Chow group of an Enriques Surfaces of Lieberman type over any field of characteristic different from $2$ or $3$ is 16 torsion. 
	
	Keep the notation of Section \ref{sec: EnSurfPre}. Define the surface $W := A/\langle \sigma_{P,Q}\rangle$ to be the quotient of the abelian surface by the involution $\sigma_{P,Q}$. Note that by the same argument as Lemma \ref{lem: KumInv}, the action of $\sigma_{P,Q}$ on $A$ is fixed-point-free, so $W$ is smooth and the natural quotient map $A \to W$ is a finite étale map of degree 2. 
	
	\begin{prop}
		The surface $W$ is a bielliptic surface. 
	\end{prop}
	\begin{proof}
		Let $\sigma$ denote $\sigma_{P,Q}$. The surface $W$ is the quotient of $E_1\times E_2$ given by a finite group action that acts on $E_2$ by a nonzero 2-torsion point. So, it remains to show that the quotient of $E_1$ by the action of $\sigma$ is isomorphic to $\mathbb{P}^1$. First note that the quotient is smooth since the fixed points of a finite group acting on a normal domain is still normal. It also has a $k$-rational point since $E_1$ does. Therefore, it remains to show that it has genus $0$. Since the genus is invariant under field extension, it suffices to assume that the base field $k$ is algebraically closed. Consider the natural quotient $\pi: E_1 \to E_1/\langle \sigma\rangle$. Then by Riemann Hurwitz
		\begin{align*}
			2g(E_1) - 2 &= \deg(\pi)(2g(E_1/\langle \sigma\rangle) - 2) + \deg R \\
			0 &= \deg(\pi)(2g(E_1/\langle \sigma\rangle) - 2) + \deg R \\
			- 2 (2g(E_1/\langle \sigma\rangle) - 2) &= \deg R
		\end{align*}
        where $R$ is the ramification divisor on $X$. But there exists $x \in E_1(\overline{k})$ such that $2x = P$. Since $x$ is fixed by $\langle \sigma \rangle$, $\pi$ is ramified at $x$, so $\deg R > 0$. Thus, $2g(E_1/\langle \sigma\rangle) - 2 < 0$, which implies that the genus of the quotient is $0$. 
	\end{proof}

    Let us consider how all of these surfaces are related. Let $G := \langle \sigma_{P, Q}\rangle$, $\tilde{G} := \langle \tilde{\sigma}_{P, Q}\rangle$, and $H := \langle \iota \rangle$. Since $X$ is birational to the quotient of $A$ by the action of $H$, we get that $k(X) = k(A)^H$. Similarly, $k(S) = k(X)^{\tilde{G}}$, and $k(W) = k(A)^G$. Also, the action of $G$ on $k(A)$ restricts to an action on $k(A)^H$ since the actions of $H$ and $G$ commute, and this restriction is equal to the action of $\tilde{G}$. Therefore, $k(S) = k(X)^{\tilde{G}} = (k(A)^H)^G$. Therefore, $k(S) \into k(W)$. So we get a commutative diagram 
    \[
    \begin{tikzcd}
         A \arrow[dashed]{d} \arrow{r} & W \arrow[dashed]{d} \\
        X \arrow{r} & S.
    \end{tikzcd}
    \]

    \begin{lem}
        Let $Z \subset W$ denote the closed subscheme given by the image of $A[2]$ under the map $A \to W$. Let $\tilde{W}$ be the blowup of $W$ at $Z$. Then the rational map $W \dashrightarrow S$ induces a degree 2 morphism $\varphi: \tilde{W} \to S$. 
    \end{lem}
    \begin{proof}
        Let $\tilde{A}$ denote the blowup of $A$ at the sixteen 2-torsion points. Note that blowing up here commutes with both quotients of $A$. So we have that $\tilde{W} = \tilde{A}/\langle\sigma_{P,Q}\rangle$ and also $S = \tilde{A}/\langle \iota, \sigma_{P,Q}\rangle$. Therefore, in the commutative diagram
        \[
            \begin{tikzcd}
                \tilde{A} \arrow[r] \arrow[d] & \tilde{W} \arrow[d] \\
                A \arrow{r} \arrow[dashed, d] & W\arrow[dashed, d] \\
                X \arrow[r] & S,
            \end{tikzcd}
        \]
        the rational map $\tilde{A} \to X$ defines a morphism, so the map $\varphi: \tilde{W} \to S$ does as well. 
    \end{proof} 
    
	\begin{thm}\label{thm: ChUpperBound}
		The group $A_0(S)$ is 16-torsion. 
	\end{thm}
	\begin{proof}
		Consider the degree 2 morphism $\varphi: \tilde{W}\to S$. The composition
		\[
			T(S) \xrightarrow{\varphi^*} T(\tilde{W}) \xrightarrow{\varphi_*} T(S)
		\]
		is multiplication by 2. Also, recall that $A_0(S) = T(S)$ since the Albanese variety of an Enriques surface is trivial. Therefore, it suffices to show that $T(\tilde{W})$ is 8-torsion. The group $T(W)$ is 8-torsion by \cite[Theorem 1]{Gazaki}. Also the birational map $\tilde{W} \to W$ induces an isomorphism on $A_0$ by Proposition \ref{prop:ChBirInv}, and this isomorphism restricts to $T(\tilde{W})$. Thus, $T(\tilde{W})$ is $8$ torsion which completes the proof. 
	\end{proof}

	\section{Brauer Groups of K3 and Enriques Surfaces}
	
	Now that we have an upper bound on the torsion of the degree-zero subgroup of the Chow group of these Enriques surfaces, we would like to study when this group is actually nontrivial. Our main tool will be the Brauer-Manin pairing. In order to make use of this pairing, we first need to establish some results about the Brauer group of Enriques surfaces of Lieberman type.
    
    Let $k$ be a field of characteristic 0 and $Y$ a smooth projective variety over $k$. The Brauer group of $Y$ is defined to be the group of Azumaya algebras over $Y$ modulo Morita equivalence. We define the cohomological Brauer group by 
	\[
	\Br(Y) := H^2_{\et}(Y, \mathbb{G}_m).
	\]
	By a result of Gabber, the cohomological Brauer group of $Y$ agrees with the Brauer group of $Y$ in this case \cite{deJong}. Therefore, we will identify these two groups throughout. Note that by the cohomological definition of the Brauer group, it is clear that it satisfies contravariant functoriality for morphisms.
	
	\begin{rem}
		Note that the cohomological Brauer group is usually defined as the torsion subgroup of $H^2_{\et}(Y, \mathbb{G}_m)$. However, in this case, $H^2_{\et}(Y, \mathbb{G}_m)$ is already torsion \cite[Chapter 3, Example 2.22]{Milne}.
	\end{rem}
	
	The structure morphism $Y \to \Spec k$ induces a map $\Br(k) \to \Br(Y)$ and we define the subgroup of constants by
	\[
	\Br_0(Y):= \text{im}\left(\Br(k) \to \Br(Y)\right).
	\]
	Also, $\overline{Y} \to Y$ induces a morphism $\Br(Y) \to \Br(\overline{Y})$. We denote the subgroup of algebraic elements by 
	\[
	\Br_1(Y) := \ker\left(\Br(Y) \to \Br(\overline{Y})\right).
	\]
	Note that the composition $\overline{Y} \to Y \to \Spec k$ is the same as $\overline{Y} \to \Spec \overline{k} \to \Spec k$. And since $\Br(\overline{k}) = 0$, we have that the composition $\Br(k) \to \Br(Y) \to \Br(\overline{Y})$ is the zero map. Therefore, we get 
	\[
	\Br_0(Y) \subseteq \Br_1(Y) \subseteq \Br(Y).
	\]
	The quotient $\Br(Y)/\Br_1(Y)$ is known as the transcendental Brauer group and will be an important object of study in what follows. 
	
	We will need the following exact sequences which are due to Grothendieck in this form. 
	\begin{thm}\label{thm: BrPurity}(\cite[Section 3.7]{Colliot-TheleneSkorobogatov})
		Let $Y$ be a smooth quasi-projective variety and let $U \subset Y$ be a nonempty open subset of $Y$. Then there exists a functorial exact sequence 
        \[
		0 \to \Br(Y) \to \Br(U) \to  \bigoplus_D H^1(k(D),\Q/\Z)
		\]
        where $D$ ranges over all irreducible divisors in $Y\setminus U$ and $k(D)$ denotes the residue field at the generic point of $D$. 

        Now let $k(Y)$ be the function field of $Y$. Then similarly there exists an exact sequence, functorial in $Y$ 
		\[
		0 \to \Br(Y) \to \Br(k(Y)) \to  \bigoplus_{D \in Y^{(1)}} H^1(k(D),\Q/\Z).
		\]
        where $k(D)$ denotes the residue field at the point $D$.
	\end{thm}

    In the case when $k = \C$, a result of Beauville implies that the Brauer group of the Enriques surface of Lieberman type $S$ injects into the Brauer group of the K3 cover $X$ if the elliptic curves are not isogenous. 
    
	\begin{prop}\label{prop: BrStoBrXgeoinj}
		Let $k = \C$ and suppose that the elliptic curves $E_1$ and $E_2$ are not isogenous. Let $X$ denote the Kummer surface associated to $A = E_1\times E_2$ and let $S$ be the Enriques surface of Lieberman type associated to the involution $\sigma_{P_0,Q_0}$ for fixed $P_0 \in E_1[2]$ and $Q_0 \in E_2[2]$. We denote the map $\pi : X \to S$. Then the morphism
		\[
		\pi^* : \Br(S) \to \Br(X)
		\]
		is injective.
	\end{prop}
	\begin{proof}
        Denote the involution $\sigma_{P_0,Q_0}$ on $A$ by $\sigma$ and $\tilde{\sigma}_{P_0, Q_0}$ on $X$ by $\tilde{\sigma}$. By \cite[Corollary 5.7]{Beauville}, it suffices to show that the Kummer surface does not have a divisor class $L \in \Pic(X)$ such that $\tilde{\sigma}^* L = -L$ and $L\cdot L \equiv 2 (\text{mod } 4)$ where $\cdot$ denotes the intersection product on $X$. Let us set some notation. We denote the rational map from the abelian surface to the Kummer surface by $\phi: A \dashrightarrow X$. For $P \in E_1[2]$ and $Q \in E_2[2]$ denote $D_{P,Q} \in \Pic(X)$ the image of the rational curve in $X$ obtained by blowing up the point $(P,Q)$ on the quotient of the abelian surface $A/\iota$. Let us also denote $\mathcal{E}_1, \mathcal{E}_2 \in \Pic(X)$ the classes given by the image of $E_1$ and $E_2$ in $X$. That is, let $j_i: E_i \into E_1\times E_2$ be the inclusion into the $i$th factor for $i = 1,2$. Then the rational map $E_i \dashrightarrow X$ given by $\phi \circ j_i$ extends uniquely to a morphism, and $\mathcal{E}_i$ is the image of this morphism. Then since $E_1, E_2$ are geometrically non-isogenous by assumption, $\Pic(X)$ is the free abelian group generated by the 18 elements $\{D_{P,Q}\}_{(P,Q) \in E_1[2]\times E_2[2]}, \mathcal{E}_1, \mathcal{E}_2$ \cite[Proposition 3.1]{Shioda}. Let 
		\[
		L = m_1\mathcal{E}_1 + m_2 \mathcal{E}_2 + \sum_{(P,Q) \in E_1[2]\times E_2[2]} n_{P,Q} D_{P,Q}
		\] 
		be an arbitrary element of $\Pic(X)$ for some $m_1, m_2, n_{P,Q} \in \Z$ and suppose $\tilde{\sigma}^*(L) = -L$. 
        
        We first claim that $\tilde{\sigma}^*$ fixes $\mathcal{E}_1$ and $\mathcal{E}_2$. Consider the projection to the second factor $p_2: E_1\times E_2 \to E_2$. Then since $2Q_0 - 2O_{E_2} = 0 \in \Pic(E_2)$, we have that $p_2^*(2Q_0 - 2O_{E_2}) = 2(E_1\times Q_0 - E_1 \times O_{E_2}) = 0\in\Pic(A)$. Also since $\sigma$ is an involution, $\sigma^*(E_1\times O_{E_2}) = \sigma(E_1\times O_{E_2}) = E_1\times Q_0$. Therefore, $2(\sigma^*(E_1\times O_{E_2}) - E_1\times O_{E_2}) = 0 \in \Pic(A)$. We also have a map $\psi: \Pic(A) \to \Pic(X)$ which is defined as the composition
        \[
            \Pic(A) \to \Pic(\tilde{A}) \to \Pic(X)
        \]
        where the first map is the pullback and the second map is the pushforward. 
        
        Now by the definition of pushforward, $\psi(E_1\times O_{E_2}) = d\mathcal{E}_1$ where $d = [k(E_1\times O_{E_2}): k(\mathcal{E}_1)]$. Moreover, $\psi(\sigma^*(E_1\times O_{E_2}))$ is by definition $[k(\sigma(E_1\times O_{E_2})): k(\phi(\sigma(E_1\times O_{E_2})))] \phi(\sigma(E_1\times O_{E_2}))$. The commutativity of 
        \[
            \begin{tikzcd}
                E_1\times \{0\} \arrow[r] & A \arrow[dashed, d, "\phi"] \arrow[r, "\sigma"] & A \arrow[dashed, d, "\phi"] \\
                & X \arrow[r, "\tilde{\sigma}"] & X
            \end{tikzcd}
        \]
        implies that $\phi(\sigma(E_1\times O_{E_2})) = \tilde{\sigma}(\mathcal{E}_1) = \tilde{\sigma}^*(\mathcal{E}_1)$ and the fact that $\sigma$ is an automorphism implies that $[k(\sigma(E_1\times O_{E_2})): k(\phi(\sigma(E_1\times O_{E_2})))] = d$. Therefore, $\psi(\sigma^*(E_1\times O_{E_2})) = d\tilde{\sigma}^*(\mathcal{E}_1)$. Therefore,
        \[
            0 = \psi(2(\sigma^*(E_1\times O_{E_2}) - E_1\times O_{E_2})) = 2d\tilde{\sigma}^*(\mathcal{E}_1) - 2d\mathcal{E}_1.
        \]
        Since the Picard group of $X$ is torsion free, we get that $\tilde{\sigma}^*(\mathcal{E}_1) = \mathcal{E}_1$. And the same argument works for $\mathcal{E}_2$.   
        Since $\tilde{\sigma}^*$ fixes $\mathcal{E}_1$ and $\mathcal{E}_2$, we get that $a_1 = a_2 = 0$. Therefore, we get that
		\[
		\tilde{\sigma}^*(L) = \sum_{(P,Q)\in E_1[2]\times E_2[2]} n_{P,Q}D_{P+P_0, Q+Q_0} = \sum_{(P,Q)\in E_1[2]\times E_2[2]} -n_{P,Q}D_{P, Q} = - L,
		\]
		which implies $-n_{P,Q} = n_{P+P_0, Q+Q_0}$ for all $P, Q$. We also have \[
		L\cdot L = \sum_{(P,Q)\in E_1[2]\times E_2[2]} -2n_{P,Q}
		\] 
		since $D_{P,Q}\cdot D_{P',Q'} = -2$ if $(P,Q) = (P',Q')$ and $0$ otherwise. 
		Let $(P_i,Q_i)$ for $1 \le i \le 8$ be representatives for the action on $E_1[2]\times E_2[2]$ given by translation by $(P_0, Q_0)$. Then 
		\begin{align*}
			L\cdot L &= \sum_{i=1}^8 -2(n_{P_i,Q_i} + n_{P_i+P_0,Q_i+Q_0} )\\
			&= \sum_{i=1}^8 -2(n_{P_i,Q_i} - n_{P_i,Q_i} ) = 0.
		\end{align*}
		Therefore, $L\cdot L$ is not equivalent to $2$ mod $4$, so the map on Brauer groups is injective by Beauville's result. 
	\end{proof}	

	\begin{rem}\label{rem: SizeofTransBr}
		Suppose now that $k$ is a number field contained in $\Q_p$ for some prime $p$ and the elliptic curves $E_1$, $E_2$ are not isogenous after base change to $\overline{k}$. By \cite[(3.3) in Section 3]{Beauville}, $\Br(S_\C) \cong \Z/2\Z$. We have by \cite[Proposition 4.2.2]{Colliot-TheleneSkorobogatov} that $\Br(S_\C) \cong \Br(S_{\overline{\Q}}) \cong \Br(S_{\overline{\Q}_p})$. Therefore, the above lemma implies that over the basefield $k$, the map $\frac{\Br(S)}{\Br_1(S)} \to \frac{\Br(X)}{\Br_1(X)}$ is injective. Also, since $\frac{\Br(S)}{\Br_1(S)} \into \Br(S_{\overline{\Q}}) \cong \Z/2\Z$, we get that $\frac{\Br(S)}{\Br_1(S)}$ is either trivial or order 2.
	\end{rem}
	
	We now need some results about the relationship between the Brauer group of the abelian surface $A$ and the associated Kummer surface $X$. This is studied in \cite{SkorobogatovZarhin}, and we recall some results from that paper which we will use. 

    First let us set some notation. By the assumption that $E_1$ and $E_2$ have rational 2-torsion, they have Weierstrass equations of the form 
    \[
        E_1: y^2 = x(x-a)(x-b)
    \]
    \[
        E_2: v^2 = u(u-a')(u-b').
    \]
    We also assume from here on that they are not isogenous over the algebraic closure $\overline{k}$. Let $\mu \in \{0, a, b\}$ and $\eta \in \{0, a', b'\}$. Then define the following quaternion algebras over the function field $k(A)$:
	\[
	\mathcal{A}_{\mu, \eta} := \left(\frac{x(x-a)(x-b)}{x-\mu},\frac{u(u-a')(u-b')}{u-\eta}\right).
	\]
	
	These define elements of $\Br(A)$ which can be shown using the exact sequence from Theorem \ref{thm: BrPurity} or an explanation can be found in \cite[pg 15]{SkorobogatovZarhin}.
	
	\begin{rem}
		We use a different notation than is used in \cite{SkorobogatovZarhin}. The change in notation amounts to relabeling the Azumaya algebras defined above. 
	\end{rem}
	
	\begin{lem}\label{lem:PresBrA}\cite{SkorobogatovZarhin}
		The map $h: E_1[2]\otimes E_2[2] \to \Br(A)[2]/\Br_1(A)[2]$ defined on simple tensors by $h((\mu,0)\otimes (\eta, 0)) = \mathcal{A}_{\mu,\eta}$ is an isomorphism.
	\end{lem}
	\begin{proof}
		Skorobogatov and Zarhin show that the $\mathcal{A}_{\mu,\eta}$ generate the group on the right \cite[pgs 14,15]{SkorobogatovZarhin}. Moreover, these groups have the same order in this case \cite[Proposition 3.3]{SkorobogatovZarhin}. So it suffices to show that $h$ defines a group homomorphism. That is, we must check that 
		\[
		\mathcal{A}_{\mu_1,\eta}\otimes_{k(A)} \mathcal{A}_{\mu_2,\eta} =  \mathcal{A}_{\mu_3,\eta}
		\]
		and
		\[
		\mathcal{A}_{\mu,\eta_1}\otimes_{k(A)} \mathcal{A}_{\mu,\eta_2} =  \mathcal{A}_{\mu,\eta_3}
		\]
		for $\mu_i$ distinct elements of $\{0,a,b\}$ and $\eta_i$ distinct elements of $\{0,a',b'\}$. Recall that for $r,s,t \in k(A)^\times$, we have the following relations in the 2-torsion of the Brauer group $\Br(k(A))[2]$: $(r,s)\otimes_{k(A)}(t,s) = (rt, s)$ and $(r^2, s) = 1$. 
        To prove the first equality of Azumaya algebras above, we have that for $\eta \in \{0, a', b'\}$ and $\mu_1, \mu_2, \mu_3$ distinct elements of $\{0,a,b\}$,
		\begin{align*}
			&\left(\frac{x(x-a)(x-b)}{x-\mu_1},\frac{u(u-a')(u-b')}{u-\eta}\right)\otimes_{k(A)} \left(\frac{x(x-a)(x-b)}{x-\mu_2},\frac{u(u-a')(u-b')}{u-\eta}\right) \\
			&= \left((x-\mu_2)(x-\mu_3),\frac{u(u-a')(u-b')}{u-\eta}\right)\otimes_{k(A)} \left((x-\mu_1)(x-\mu_3),\frac{u(u-a')(u-b')}{u-\eta}\right) \\
			&=  \left((x-\mu_1)(x-\mu_2)(x-\mu_3)^2,\frac{u(u-a')(u-b')}{u-\eta}\right)\\
			&=  \left((x-\mu_1)(x-\mu_2),\frac{u(u-a')(u-b')}{u-\eta}\right) = \mathcal{A}_{\mu_3, \eta}.
		\end{align*}
		And the second equality follows by a similar computation. 
	\end{proof}	
	
	Recall that we denote by $\iota: A \to A$ the negation involution on the abelian surface. Then the functions 
	\[
	\frac{x(x-a)(x-b)}{x-\mu}, \frac{u(u-a')(u-b')}{u-\eta}
	\]
	are fixed by the action of $\iota$ on $k(A)$. Therefore, the Brauer classes $\mathcal{A}_{\mu, \eta}$ define elements of $\Br(k(X))$. We also have the following result which establishes when algebras lift to Brauer classes on the Kummer surface $X$.
	
	\begin{lem}\label{lem:BrClassesDescendingToKum}(\cite{SkorobogatovZarhin} Lemma 3.6)
		The sequence 
		\[
		0 \to \Br(X)[2]/\Br_1(X)[2] \to \Br(A)[2]/\Br_1(A)[2] \xrightarrow{\Phi} k^\times/(k^\times)^2
		\]
		is exact where the map $\Phi$ is defined by
		\begin{align*}
			&\Phi(\mathcal{A}_{0,0}) = (1, ab, a'b', -aa') \\
			&\Phi(\mathcal{A}_{a,0}) = (a'b', aa', 1, a(a-b)) \\
			&\Phi(\mathcal{A}_{0,a'}) = (ab, 1, aa', a'(a'-b')) \\
			&\Phi(\mathcal{A}_{a,a'}) = (-aa', a'(a'-b'), a(a-b), 1) \\
		\end{align*}
		and then extended to a group homomorphism.
	\end{lem}

    Recall that for an Enriques surface defined over $k$, the transcendental Brauer group $\Br(S)/\Br_1(S)$ is either trivial or $\Z/2\Z$ (see Remark \ref{rem: SizeofTransBr}). We now want to determine precisely when this group is nontrivial. Note that the field $k(A)$ is generated by the rational functions $x,y,u,v$ over $k$. Since the action of $\iota$ sends $y$ to $-y$, $v$ to $-v$ and fixes $x$ and $u$, we have that $x, u \in k(X) = k(A)^\iota$. Let $P = (\mu, 0)$ and $Q = (\eta,0)$ where $\mu \in \{0,a, b\}$ and $\eta \in \{0, a', b'\}$. Let $\mu_2, \mu_3$ denote the other elements that are not equal to $\mu$ and similarly $\eta_2, \eta_3$ are the elements not equal to $\eta$. Then 
    \[
        \tilde{\sigma}_{P,Q}^*(x) = \sigma_{P,Q}^*(x)= \mu + \frac{(\mu - \mu_2)(\mu-\mu_3)}{x-\mu}
    \]
    where the second equality follows by the formula for the translation by the point $(\mu, 0)$ on the elliptic curve $E_1$. Similarly, 
    \[
        \tilde{\sigma}_{P,Q}^*(u) =\eta + \frac{(\eta - \eta_2)(\eta-\eta_3)}{u-\eta}.
    \]
    Note this implies 
    \[
        \tilde{\sigma}_{P,Q}^*\left(\frac{(\mu - \mu_2)(\mu-\mu_3)}{x-\mu}\right) = x-\mu
    \]
    and similarly for $u$ which will be relevant in the proof of the following theorem.
    
    The proof of the following theorem is similar to some of the arguments in section 3 of \cite{SkorobogatovZarhin} in the sense that we define Brauer classes in the function field of the surface $S_{P,Q}$ and then determine when these lift to Brauer classes on $S_{P,Q}$ by studying their ramification.

	\begin{thm}\label{thm: CalcBrS}
		Let $S := S_{P,Q}$ be the Enriques surface of Lieberman type obtained from $X$ by quotienting by $\tilde{\sigma}_{P,Q}$. Let $P=(\mu, 0)$ and $Q = (\eta, 0)$ for $\mu \in \{0,a,b\}$ and $\eta \in \{0, a',b'\}$. Then $\Br(S)/\Br_1(S)$ is nontrivial if and only if $\mathcal{A}_{\mu, \eta}$ is in the image of the map
		\[
		\Br(X)[2]/\Br_1(X)[2] \to \Br(A)[2]/\Br_1(A)[2].
		\]
		Moreover, in this case, the image of the nontrivial element under the composition 
		\[
		\Br(S)/\Br_1(S) \to \Br(X)[2]/\Br_1(X)[2] \to \Br(A)[2]/\Br_1(A)[2]
		\]
		is $\mathcal{A}_{\mu, \eta}$. 
	\end{thm}
	
	\begin{proof}
		First, let us assume that $A_{\mu, \eta}$ is in the image of the map 
        \[
            \Br(X)[2]/\Br_1(X)[2] \to \Br(A)[2]/\Br_1(A)[2].
        \]
        Let $\mu_2, \mu_3$ denote the two elements of the set $\{0,a,b\}$ which are not equal to $\mu$ and $\eta_2, \eta_3$ denote the two elements of the set $\{0,a',b'\}$ which are not equal to $\eta$. Define the following two elements of $\overline{k(A)}$, the algebraic closure of the function field of the abelian surface:
		\[
		f_\mu := (x-\mu) + \frac{(\mu - \mu_2)(\mu-\mu_3)}{(x-\mu)} - 2\sqrt{(\mu - \mu_2)(\mu-\mu_3)} 
		\]
		\[
		g_\eta := (u-\eta) + \frac{(\eta - \eta_2)(\eta - \eta_3)}{(u-\eta)} - 2\sqrt{(\eta - \eta_2)(\eta - \eta_3)}
		\]
		where the square roots denote some fixed square root in the algebraic closure. Let us consider the quaternion algebra $(f_\mu, g_\eta) \in \Br(\overline{k(A)})$. Our goal is to show that $(f_\mu, g_\eta)$ is in the image of the natural map 
        \[
            \Br(S) \to \Br(k(S)) \to \Br(k(A)) \to \Br(\overline{k(A)}).
        \]
        We do this in several steps. 
        \\ \underline{Step 1: $(f_\mu, g_\eta) \in \Br(k(S))$:} \\
        We need to show that $f_\mu, g_\eta \in k(S)$. By Lemma \ref{lem:BrClassesDescendingToKum} and by the assumption on $\mathcal{A}_{\mu,\eta}$, we have that $\Phi(\mathcal{A}_{\mu, \eta}) =0$. It is easy to check by the definition of $\Phi$ and examining each case that this implies $(\mu - \mu_2)(\mu - \mu_3)$ and $(\eta-\eta_2)(\eta-\eta_3)$ are squares in $k$. For example, consider the case of $\mu = 0$, $\eta = 0$. Then by the definition of $\Phi$, $\Phi(\mathcal{A}_{0,0})$ implies $ab = (0 -a)(0-b)$ is a square and $a'b' = (0 - a')(0-b')$ is a square. Therefore, $f_\mu, g_\eta \in k(A)$. The action induced by $\iota : A \to A$ on $k(A)$ fixes $f_\mu, g_\eta$ since they are functions of $x$ and $u$ respectively. Therefore, $f_\mu, g_\eta$ are elements of $k(X)$. Moreover, the induced action of $\tilde{\sigma}_{P,Q}$ on $k(X)$ fixes $f_\mu, g_\eta$:
        \begin{align*}
            \tilde{\sigma}_{P,Q}^*(f_\mu) &= \sigma_{P,Q}^*\left((x-\mu) + \frac{(\mu - \mu_2)(\mu-\mu_3)}{(x-\mu)} - 2\sqrt{(\mu - \mu_2)(\mu-\mu_3)}\right) \\
            &= \sigma_{P,Q}^*(x-\mu) + \sigma_{P,Q}^*\left(\frac{(\mu - \mu_2)(\mu-\mu_3)}{(x-\mu)}\right) - 2\sqrt{(\mu - \mu_2)(\mu-\mu_3)}\\
            &= (x-\mu) + \frac{(\mu - \mu_2)(\mu-\mu_3)}{(x-\mu)} - 2\sqrt{(\mu - \mu_2)(\mu-\mu_3)} = f_\mu
        \end{align*}
        and the same argument applies for $g_\eta$. Therefore, $f_\mu, g_\eta \in k(S)$ which completes step 1.
        \\\underline{Step 2:}\\
        Let $U := A\setminus A[2]$ be the open set of $A$ which is the complement of the sixteen 2-torsion points of $A$. Then we have a map $U \to S$, and let the open $W \subset S$ be the image. For this step we want to show that $(f_\mu, g_\eta)$ is in the image of $\Br(W) \to \Br(k(S))$. The map $\psi: U \to W$ is unramified, and Theorem \ref{thm: BrPurity} implies that we get the following commutative diagram with exact rows
		\[
		\begin{tikzcd}
			0 \arrow{r} & \Br(U) \arrow{r} & \Br(k(A)) \arrow{r} & \bigoplus_{D \in U^{(1)}} H^1(k(D), \Q/\Z) \\
			0 \arrow{r} & \Br(W) \arrow{r} \arrow{u}{\psi^*} & \Br(k(S)) \arrow{r}\arrow{u} & \bigoplus_{Z \in W^{(1)}} H^1(k(Z), \Q/\Z) \arrow{u} .
		\end{tikzcd}
		\]
		We will show that the image of $(f_\mu, g_\nu)$ under the bottom right horizontal map is $0$. Since this element is in $\Br(k(S))[2]$, we can instead consider the diagram
		\[
		\begin{tikzcd}
			0 \arrow{r} & \Br(U)[2]\arrow{r} & \Br(k(A))[2] \arrow{r} & \bigoplus_{D \in U^{(1)}} H^1(k(D), \Q/\Z)[2] \\
			0 \arrow{r} & \Br(W)[2] \arrow{r} \arrow{u} & \Br(k(S))[2] \arrow{r}\arrow{u} & \bigoplus_{Z \in W^{(1)}} H^1(k(Z), \Q/\Z)[2] \arrow{u} .
		\end{tikzcd}
		\]
		Now $H^1(k(Z), \Q/\Z)[2] \cong H^1(\Gal(k(Z)), \Z/2\Z)$ and by the Kummer exact sequence, the latter is isomorphic to $k(Z)^\times/(k(Z)^\times)^2$. Moreover, for the class of a quaternion algebra $(f,g) \in \Br(k(S))$, the image of $(f,g)$ in $k(Z)^\times/(k(Z)^\times)^2$ is given by 
		\[
		(-1)^{v_Z(f)v_Z(g)}\frac{f^{v_Z(g)}}{g^{v_Z(f)}} \in k(Z)^\times/(k(Z)^\times)^2
		\]
		where $v_Z$ denotes the valuation with respect to $Z$. Thus it suffices to show that $v_Z(f_\mu), v_Z(g_\eta)$ are even for each $Z \in W^{(1)}$. However, since $U \to W$ is unramified, for any $D \in U^{(1)}$ mapping to $Z \in W^{(1)}$, for any rational function $f$ on $S$, $v_Z(f) = v_D(f)$ where on the right hand side $f$ is considered as a rational function on $A$. Therefore, we need to check that the valuations of $f_\mu$ and $g_\eta$ for every codimension $1$ point on $A$ are even. However we note that in $k(A)$, 
		\begin{align*}
			&\frac{1}{x-\mu}\left((x-\mu) - \sqrt{(\mu - \mu_2)(\mu-\mu_3)}\right)^2 \\
            &= (x-\mu) + \frac{(\mu - \mu_2)(\mu-\mu_3)}{(x-\mu)} - 2\sqrt{(\mu - \mu_2)(\mu-\mu_3)} = f_\mu
		\end{align*}
		and 
		\begin{align*}
			&\frac{1}{u-\eta}\left((u-\eta) - \sqrt{(\eta - \eta_2)(\eta-\eta_3)}\right)^2 \\
            &= (u-\eta) + \frac{(\eta - \eta_2)(\eta-\eta_3)}{(u-\eta)} - 2\sqrt{(\eta - \eta_2)(\eta-\eta_3)} = g_\eta
		\end{align*}
        by expanding the above products. Therefore, it suffices to show that $x-\mu$, $u-\eta$ have even valuation at every codimension 1 point of $A$. To do so we compute $\dv(x-\mu)$ and $\dv(u-\eta)$. Let $\pi_1 : E_1\times E_2 \to E_1$ be the projection onto the first factor. By definition, the function $x-\mu$ in $k(A)$ is an element of $k(E_1) \subset k(A)$. Therefore, we have 
        \begin{align*}
            \dv_A(x-\mu) &= \pi_1^*(\dv_{E_1}(x-\mu)) \\
            &= \pi_1^*(2(\mu, 0) - 2 O_{E_1}) \\
            &= 2(\mu, 0)\times E_2 - 2 O_{E_1}\times E_2.
        \end{align*}
        Therefore, the valuation of $x-\mu$ is even at every codimension $1$ point of $A$. The same argument applies to $u-\eta$ as well. Therefore, $(f_\mu, g_\eta)$ defines a Brauer class of $\Br(W)$ which completes Step 2.
        \\\underline{Step 3:} \\
        We now show that $(f_\mu, g_\eta)$ is in the image of the map $\Br(S) \to \Br(W)$. Let $V \subset X$ be the open subset given by the image of $U \to X$. That is, the complement of the 16 lines given by blowing up the singularities on the quotient of the abelian surface. Denote these 16 lines by $L_{i}$ for $1 \le i \le 16$. And denote the set of their images in $S$ by $H_j$ for $1\le j\le 8$. Now by Theorem \ref{thm: BrPurity} we have the following commutative diagram with exact rows 
		\[
		\begin{tikzcd}
			0 \arrow{r} & \Br(X)[2]\arrow{r} & \Br(V)[2] \arrow{r} & \bigoplus_{1\le i \le 16}k(L_i)^\times/(k(L_i)^\times)^2 \\
			0 \arrow{r} & \Br(S)[2] \arrow{r} \arrow{u} & \Br(W)[2] \arrow{r}\arrow{u} & \bigoplus_{1\le j\le 8} k(H_j)^\times/(k(H_j)^\times)^2\arrow{u} .
		\end{tikzcd}
		\]
		We again want to show that the image of $(f_\mu, g_\eta)$ under the bottom right vertical map is $0$. We know by assumption that its image in the top right corner is $0$, so it suffices to show that the right vertical map is injective. This is true because for $L_i$ lying over $H_j$, the map $k(H_j) \to k(L_i)$ is an isomorphism. Therefore, we have that $(f_\mu, g_\eta)$ defines an element of $\Br(S)$. This completes Step 3. 
        
        By the above, $(f_\mu, g_\eta)$ is in the image of the map $\Br(S) \to \Br(A) \to \Br(k(A))$. It therefore remains to show that $(f_\mu, g_\eta) = \mathcal{A}_{\mu,\eta}$ in $\Br(k(A))$ as $\Br(A) \to \Br(k(A))$ injective.  From step 2, its image in $\Br(k(A))$ is equal to $(\frac{1}{x-\mu}, \frac{1}{u-\eta})$ but since $y^2 = (x-\mu)(x-\mu_2)(x-\mu_3)$ and $v^2 = (u-\eta)(u-\eta_2)(u-\eta_3)$  we get that the above Brauer class is equal to $((x-\mu_2)(x-\mu_3), (u-\eta_2)(u-\eta_3))$ which is $\mathcal{A}_{\mu, \eta}$ by definition. This completes the first implication of the theorem. 
		
		To prove the reverse implication, we must show that if $A_{\mu, \eta}$ is not in the image of $\Br(X)/\Br_1(X)$, then $\Br(S)/\Br_1(S) = 0$. Consider the following commutative diagram:
		\[
		\begin{tikzcd}
			\Br(S_{\overline{k}}) \arrow{r} & \Br(X_{\overline{k}})[2] \arrow{r}{\cong} & \Br(A_{\overline{k}})[2] \\
			\Br(S)/\Br_1(S) \arrow{r}  \arrow{u} & \Br(X)[2]/\Br_1(X)[2] \arrow{r} \arrow{u} & \Br(A)[2]/\Br_1(A)[2] \arrow{u}{\cong}.
		\end{tikzcd}
		\]
        The isomorphism for the right vertical arrow follows from \cite[Proposition 3.3]{SkorobogatovZarhin}, the fact that the 2-torsion of $E_1$ and $E_2$ are defined over the base field, so the Galois group acts trivially, and the assumption that the elliptic curves are not isogenous over the algebraic closure. The isomorphism for the top right horizontal map follows from Lemma \ref{lem:BrClassesDescendingToKum} and the fact that $\overline{k}^\times/(\overline{k}^\times)^2$ is trivial. 

        The argument from the first part of the proof works over $\overline{k}$, so the image of the nontrivial element of $\Br(S_{\overline{k}})$ in $\Br(A)[2]/\Br_1(A)[2]$ is $\mathcal{A}_{\mu, \eta}$. Therefore, by commutativity, if the left vertical map was an isomorphism, the nontrivial element of $\Br(S)/\Br_1(S)$ would have to map to $\mathcal{A}_{\mu, \eta}$ which contradicts the assumption that it is not in the image of $\Br(X)[2]/\Br_1(X)[2]$.  This completes the proof.
	\end{proof}
	
	\begin{rem}
		Note that the above theorem and the previous lemmas from \cite{SkorobogatovZarhin} imply that we can check whether the transcendental Brauer group of $S$ is nontrivial and determine its image by just checking a few square conditions of terms given by the coefficients of the elliptic curves. 
	\end{rem}

    Now that we can calculate the transcendental Brauer group of the Enriques surface, our goal in the next section will be to study the Chow group over p-adic fields using the Brauer-Manin pairing. This will allow us to find nonzero elements of the degree 0 part of the Chow group of the Enriques surface. Before doing so, we will need one more result which we use in the next section concerning the Brauer group of the abelian surface. Proposition 3.3 of \cite{SkorobogatovZarhin} gives a canonical isomorphism
    \[
        \Br(A)[2]/\Br_1(A)[2]\cong \frac{\Hom_{\Gal(k)}(E_1[2], E_2[2])}{(\Hom(\overline{E_1}, \overline{E_2})/2)^{\Gal(k)}}.
    \]
    However, by assumption, $E_1$ and $E_2$ are not isogenous over the algebraic closure and the 2-torsion of both elliptic curves is defined over $k$. Therefore, the above reduces to a canonical isomorphism
    \[
        \Br(A)[2]/\Br_1(A)[2]\cong \Hom(E_1[2], E_2[2]).
    \]
    We need to compute this isomorphism explicitly which is the following lemma.

    \begin{lem}\label{lem: SZCanonicalIso}
        The canonical isomorphism 
        \[
	       \Br(A)[2]/\Br_1(A)[2] \to \Hom(E_1[2], E_2[2])
	    \]
        is given by $\mathcal{A}_{\mu, \eta} \mapsto \varphi_{\mu, \eta}$ where 
	\begin{align*}
		\varphi_{\mu, \eta}(O_{E_1}) &= O_{E_2} \\
		\varphi_{\mu, \eta}((\mu, 0)) &= O_{E_2} \\
		\varphi_{\mu, \eta}((\mu_2,0)) &= \eta\\
		\varphi_{\mu, \eta}((\mu_3,0)) &= \eta.
	\end{align*}
    \end{lem}
    Note: Throughout this proof we will often write the group $\{\pm 1\} \subset k$ as $\Z/2\Z$ rather than $\mu_2$ to avoid notational confusion with the 2-torsion points on $E_1$.
    \begin{proof}
        We consider the canonical map 
        \[
            E_1[2]\otimes E_2[2] \to \Hom(E_1[2],E_2[2]) \to \Br(A)[2]/\Br_1(A)[2]
        \]
        where the first map is the canonical map induced by the Weil pairing and the second map is the canonical isomorphism given in \cite{SkorobogatovZarhin}. We first claim that the map $E_1[2]\otimes E_2[2] \to \Hom(E_1[2],E_2[2])$ is given by $((\mu, 0),(\eta,0)) \mapsto \varphi_{\mu, \eta}$. 
        
        Consider the Weil pairing
        \[
            w(- , - ): E_1[2]\times E_1[2] \to \{\pm 1\}
        \]
        which is alternating, bilinear, and nondegenerate. By these properties, it is easy to see that $w(P,Q) = -1$ if and only if $P\neq Q$ and neither is the trivial group element. Therefore, the induced isomorphism
        \[
            E_1[2] \to \Hom(E_1[2], \Z/2\Z)
        \]
        is defined by $P \mapsto \psi_P$ where $\psi_P$ is the homomorphism which sends $0$ and $P$ to $0$ and the two nontrivial elements not equal to $P$ map to $1$. So the composition
        \[
            E_1[2]\otimes E_2[2] \cong \Hom(E_1[2], \Z/2\Z)\otimes E_2[2] \cong \Hom(E_1[2], E_2[2])
        \]
        is given by $((\mu, 0), (\eta, 0)) \mapsto \psi_{(\mu, 0)}(-)(\eta, 0)$. Now we get $\psi_{(\mu, 0)}(O_{E_1}) = \psi_{(\mu, 0)}((\mu, 0))= 0$ and $\psi_{(\mu, 0)}((\mu_2, 0)) = \psi_{(\mu_3, 0)}((\mu, 0))= 1$. Thus the function $\psi_{(\mu, 0)}(-)(\eta, 0) = \varphi_{\mu, \eta}$ by definition. So to prove the Lemma we must show that 
        \[
            E_1[2]\otimes E_2[2] \to \Hom(E_1[2],E_2[2]) \to \Br(A)[2]/\Br_1(A)[2]
        \]
        sends $((\mu,0),(\eta, 0))$ to $\mathcal{A}_{\mu, \eta}$. We also can compose with the injections 
        \[
            \Br(A)[2]/\Br_1(A)[2] \to \Br(\overline{A})[2] \to \Br(\overline{k}(A))[2]
        \]
        and so it suffices to show that the map 
        \[
            E_1[2]\otimes E_2[2] \to \Hom(E_1[2],E_2[2]) \to \Br(A)[2]/\Br_1(A)[2] \to \Br(\overline{k}(A))[2]
        \]
        sends $((\mu, 0), (\eta, 0)) \mapsto \mathcal{A}_{\mu,\eta}$. 
        
        The canonical isomorphism in \cite[Section 3]{SkorobogatovZarhin} is described in the following way. Let $p_1: A \to E_1$, $p_2: A \to E_2$ be the projections to the first and second factors. Let $\mathcal{T}_1$ and $\mathcal{T}_2$ be the $E_1$ and $E_2$ torsors with structure groups $E_1[2]$ and $E_2[2]$ respectively defined by the multiplication by 2 maps $[2]: E_i \to E_i$. Then let $[\mathcal{T}_1] \boxtimes [\mathcal{T}_2] \in H_{\et}^2(A, E_1[2]\otimes E_2[2])$ denote $p_1^*(\mathcal{T}_1) \cup p_2^*(\mathcal{T}_2)$. Since the 2-torsion of both elliptic curves is defined over $k$, for $\varphi \in \Hom(E_1[2]\otimes E_2[2], \mu_2)$, there is an induced map $\varphi_*: H^2_{\et}(A, E_1[2]\otimes E_2[2]) \to H^2(A, \mu_2)$. Denote the map 
        \[
            \xi: \Hom(E_1[2]\otimes E_2[2], \Z/2\Z) \to H^2(A, \Z/2\Z)
        \]
        to be $\xi(\varphi) = \varphi_*([\mathcal{T}_1]\boxtimes[\mathcal{T}_2])$. Then by \cite[Proposition 3.3]{SkorobogatovZarhin}, the canonical isomorphism is induced by the composition
        \[
            \Hom(E_1[2], E_2[2]) \to \Hom(E_1[2]\otimes E_2[2], \Z/2\Z) \xrightarrow{\xi} H^2(A, \Z/2\Z) \to \Br(A)[2]
        \]
        where the first map is the canonical isomorphism induced by the Weil pairing and the last map is the one from the Kummer sequence associated to 
        \[
            0 \to \Z/2\Z \to \G_m \to \G_m \to 0.
        \]

        Now let us consider the composition of the above with the map $\Br(A)[2] \to \Br(\overline{A})[2]$. By \cite{SkorobogatovZarhin} Proposition 3.3 and the remarks on page 12, the composition
        \[
            \Hom(E_1[2] \otimes E_2[2], \Z/2\Z) \xrightarrow{\xi} H^2(A, \Z/2\Z) \to \Br(A)[2] \to \Br(\overline{A})[2]
        \]
        is equal to the composition
        \begin{align*}
            &\Hom(E_1[2] \otimes E_2[2], \Z/2\Z) \to \Hom(E_1[2], \Z/2\Z) \otimes \Hom(E_2[2], \Z/2\Z) \xrightarrow{\tau_1\otimes \tau_2} \\
            &H^1(\overline{E_1}, \Z/2\Z)\otimes H^1(\overline{E_2}, \Z/2\Z) \xrightarrow{p_1^*\otimes p_2^*} H^1(\overline{A}, \Z/2\Z)\otimes H^1(\overline{A},\Z/2\Z) \xrightarrow{\cup} H^2(\overline{A}, \Z/2\Z) \to \Br(\overline{A})[2]
        \end{align*}
        where the first map is induced by the Weil pairing, the last map is the one arising from the Kummer sequence, and the maps $\tau_1$ and $\tau_2$ are defined as follows. Let $[\overline{\mathcal{T}}_i]$ be the image of $[\mathcal{T}_i]$ under the natural map $H^1(E_i, E_i[2]) \to H^1(\overline{E}_i, E_i[2])$. Then $\tau_i(\varphi) := \varphi_*([\overline{\mathcal{T}}_i])$.
        \\
        \underline{Claim (*)}: Consider the map 
        \[
            E_1[2]\otimes E_2[2] \to\Hom(E_1[2]\otimes E_2[2], \Z/2\Z) \to \Br(\overline{A})[2] \to \Br(\overline{k}(A))[2]
        \]
        where the first map is the canonical morphism induced by the Weil pairing and the second map is the above composition. Then this is equal to 
        \[
            E_1[2]\otimes E_2[2] \to \Br(\overline{k}(A))[2]
        \]
        defined by sending $P\otimes Q$ to the quaternion algebra $(f_P, g_Q)$, where $f_P$ has the divisor $2(P)-2(O_{E_1})$ and $g_Q$ has the divisor $2(Q) - 2(O_{E_2})$ (note that any two such functions with these above properties differ by a constant and so they define the same quaternion algebra in $\Br(\overline{k}(A))$).

        Let us assume Claim (*). Then if we let $P=(\mu,0)$ and $Q = (\eta,0)$, we can take $f_P = x-\mu$ and $g_Q = u-\eta$. Since modulo squares $x-\mu \equiv \frac{x(x-a)(x-b)}{x-\mu}$ and $u-\eta \equiv \frac{u(u-a')(u-b')}{u-\eta}$ we get that the map $E_1[2]\otimes E_2[2] \to \Br(\overline{k}(A))[2]$ is given by $((\mu,0),(\eta,0)\mapsto \mathcal{A}_{\mu,\eta}$. 

        Finally, we prove Claim (*). Consider the following commutative diagram 
        \[
            \begin{tikzcd}
                H_\et^1(\overline{E_1}, \Z/2\Z)\otimes H_\et^1(\overline{E_2}, \Z/2\Z) \arrow[r, "p_1^*\cup p_2^*"] \arrow{d} & H_\et^2(\overline{A}, \Z/2\Z) \arrow{d} \arrow[r] &\Br(\overline{A})[2] \arrow{d} \\
                H_\et^1(\overline{k}(E_1), \Z/2\Z)\otimes H_\et^1(\overline{k}(E_2), \Z/2\Z) \arrow[r, "p_1^*\cup p_2^*"] \arrow[d] & H_\et^2(\overline{k}(A), \Z/2\Z) \arrow[r] &\Br(\overline{k}(A))[2] \arrow[d,"="] \\
                \overline{k}(E_1)^\times/(\overline{k}(E_1)^\times)^2 \otimes \overline{k}(E_2)^\times/(\overline{k}(E_2)^\times)^2 \arrow[rr] & & \Br(\overline{k}(A))[2].
            \end{tikzcd}
        \]
        By \cite[Proposition 4.7.1]{GilleSzamuely}, the bottom horizontal map is given by $f\otimes g \mapsto (f,g)$ where the right hand side denotes the class of the quaternion algebra. Therefore, it suffices to show that the composition
        \[
            \Hom(E_1[2], \Z/2\Z) \xrightarrow{\tau_i} H_\et^1(\overline{E}_1, \Z/2\Z) \to H_\et^1(\overline{k}(E_1), \Z/2\Z) \cong \overline{k}(E_1)/((\overline{k}(E_1))^\times)^2
        \]
        sends $w(P, -)$ to $f_P$ where $f_P$ is a rational function with the divisor $2(P)-2(O_{E_i})$ and $w(-,-)$ is the Weil pairing (the argument for $E_2$ is the same). In other words, we need to show that the image of $[\overline{\mathcal{T}}_1]$ under the composition 
        \[
            H_\et^1(\overline{E}_1, E_1[2]) \to H_\et^1(\overline{k}(E_1), E_1[2]) \xrightarrow{w(P, -)_*}  H_\et^1(\overline{k}(E_1), \Z/2\Z) \to \overline{k}(E_1)/(\overline{k}(E_1)^\times)^2
        \]
        is $f_P$. Note that $\overline{k}(E_1)$ is the function field of $\overline{E}_1$ as they are both generated by the rational functions $x,y$ over $\overline{k}$. Consider the pullback diagram 
        \[
            \begin{tikzcd}
            \overline{E}_1\times_{\overline{E}_1}\Spec \overline{k}(E_1) \arrow[r] \arrow[d] & \Spec \overline{k}(E_1) \arrow[d, "\eta"] \\
            \overline{E}_1  \arrow[r, "{[2]}"] & \overline{E}_1
            \end{tikzcd}
        \]
        where $\eta \in \overline{E}_1(\overline{k}(E_1))$ corresponds to the generic point of the elliptic curve. Denote $Y := \overline{E}_1\times_{\overline{E}_1}\Spec \overline{k}(E_1)$. Note that it is naturally isomorphic to $\Spec \overline{k}(E_1)$. Then the pullback of the torsor $\overline{\mathcal{T}}_1$ to the generic point is the torsor 
        \[
            Y \to \Spec \overline{k}(E_1)
        \]
        with structure group $E_1[2]$ which acts by the map induced by translation on $Y$. Then this torsor corresponds to a cocyle in $H_\et^1(\overline{E}_1, \Z/2\Z)$ in the following way. 
        
        Choose a point $T_0 \in Y(\overline{k(E_1)})$. This corresponds to a point which we also denote $T_0 \in \overline{E}_1(\overline{k(E_1)})$ such that $[2](T_0) = \eta \in \overline{E}_1(\overline{k(E_1)})$. Then the torsor corresponds to the cocycle 
        \[
            \Gal(\overline{k(E_1)}/\overline{k}(E_1)) \to E_1[2]
        \]
        \[
            \theta \mapsto c_\theta
        \]
        where $c_\theta := \theta(T_0) - T_0 \in \overline{E}_1(\overline{k(E_1)})$. Note that $c_\theta$ indeed defines an element of $E_1[2]$ since $[2](\theta(T_0) - T_0) = \theta([2]T_0) - 2T_0 = \theta(\eta) - \eta = O_{\overline{E}_1}$ where the first equality follows from the fact that the multiplication by 2 morphism is defined over $k$ and hence over $\overline{k}(E_1)$. 
    
        Now let $f_P$ be a function in $\overline{k}(E_1)$ with divisor $2(P) - 2(O_{E_1})$. And let $h \in \overline{k}(E_1)$ be such that $h^2 = f \circ [2]$. Then by the definition of the Weil pairing, we get that $w(c_\theta, P) = \frac{h(x+c_\theta)}{h(x)}$ for all $x \in \overline{E}_1(\overline{k})$. However, the same formula holds for any $x \in \overline{E}_1(\overline{k(E_1)})$ by considering the Weil pairing for the base change of $\overline{E}_1$ to $\overline{k}(E_1)$. Therefore, if we take $x = T_0$, we get
        \[
            w(c_\theta, P) = \frac{h(T_0+\theta(T_0) - T_0)}{h(T_0)} = \frac{h(\theta(T_0))}{h(T_0)} = \frac{\theta(h(T_0))}{T_0}
        \]
        since $h$ is defined over $\overline{k}(E_1)$ and hence commutes with $\theta$. 
        
        So the image of $\overline{\mathcal{T}}_1$ in $H^1(\overline{k}(E_1), \Z/2\Z)$ under the composition
        \[
            H^1(\overline{E}_1, E_1[2]) \to H^1(\overline{k}(E_1), E_1[2]) \xrightarrow{w(P, -)_*}  H^1(\overline{k}(E_1), \Z/2\Z)
        \]
        is the cocycle defined by $\theta \mapsto \frac{\theta(h(T_0))}{h(T_0)}$. But since by definition, $f_P=h^2$, we get by the standard Kummer identification $\overline{k}(E_1)^\times/(\overline{k}(E_1)^\times)^2 \to H^1(\overline{k}(E_1), \Z/2\Z)$  that the class of $f_P$ maps to $w(c_\sigma, P)$. But this is exactly what we needed to show, so it completes the proof of Claim (*).
    \end{proof}
	
	\section{Brauer-Manin Pairing and the Chow Group Over p-adic Fields}
    Throughout this section the base field $k$ will be a finite extension of $\Q_p$.
    \subsection{The Brauer-Manin Pairing}
	Let $Y$ be a smooth proper variety over a p-adic field $k$. Then the Brauer-Manin pairing \cite[pg 53]{Colliot-Thelene}
	\[
	\langle -, - \rangle_Y: \ch_0(Y)\times \Br(Y) \to \Q/\Z
	\]
	is defined as follows. Given the class of a closed point $[y] \in \ch_0(Y)$ the closed embedding $\iota_y: \Spec k(y) \into Y$ induces a map $\iota_y^*: \Br(Y) \to \Br(k(y))$. We also have the corestriction map from Galois cohomology $\Cor_{k(y)/k}: \Br(k(y)) \to \Br(k)$. Then given a Brauer class $\alpha \in \Br(Y)$, $\langle [y], \alpha \rangle_Y := \text{inv}\left(\Cor_{k(y)/k}(\iota_y^*(\alpha))\right)$ where $\text{inv}: \Br(k) \to \Q/\Z$ is the invariant map isomorphism. We can restrict this pairing to the degree $0$ subgroup $A_0(Y)$ of $\ch_0(Y)$, and we get the pairing
	\[
	\langle -, - \rangle_Y: A_0(Y)\times \Br(Y) \to \Q/\Z.
	\]
	Notice that for an element $\alpha \in \Br_0(Y)$, pairing with a cycle $z \in \ch_0(X)$ amounts to multiplication by the degree of the cycle $z$. Therefore, when restricting to the degree $0$ subgroup, we get a well-defined pairing
	\[
	\langle -, - \rangle_Y: A_0(Y)\times \Br(Y)/\Br_0(Y) \to \Q/\Z.
	\]
	Restricting to the subgroup $\Br(Y)[n]/\Br_0(Y)[n]$ on the right and $T(Y)$ on the left, we get a well-defined pairing 
	\[
	\langle -, - \rangle_Y: T(Y)/n\times \Br(Y)[n]/\Br_0(Y)[n] \to \Z/n\Z.
	\]
	We also have that the pairing is functorial in the following sense.
	\begin{lem}\label{lem: BMFunctoriality}\cite[Lemma 3.7]{GazakiLove}
		Let $f: Y \to Z$ be a proper map of smooth varieties over a $p$-adic field $k$. Then for any $n\ge 1$ the following diagram commutes:
		
		\[\begin{tikzcd}
			T(Y)/n \arrow[d,"f_*"] & [-3em] \times & [-3em] \Br(Y)[n]/\Br_0(Y)[n] \arrow[r,"{\langle-,-\rangle_Y}"] & \Z/n\Z \\
			T(Z)/n & [-3em] \times & [-3em] \Br(Z)[n]/\Br_0(Z)[n] \arrow[u,"f^*"] \arrow[r,"{\langle-,-\rangle_Z}"'] &\Z/n\Z \arrow[u,equal].
		\end{tikzcd}\]
	\end{lem}
    
    Let us define our setup which we will assume for our results over p-adic fields. 
    \begin{set} \label{set: PadicSetting}
        The abelian surface $A$ is the product of the elliptic curves 
        \[
            E_1: y^2 = x(x-a)(x-b)
        \]
        \[
            E_2: v^2 = u(u-a')(u-b')
        \]
        which are defined over a number field $F$. We assume that the elliptic curves are not isogenous over $\overline{\Q}$. Let $\mathfrak{p}$ be a fixed prime in $\mathcal{O}_F$ such that the residue field $\F_{\mathfrak{p}}$ has characteristic $p > 2$. Let $k$ denote the p-adic field which is the completion of $F$ with respect to the norm associated to $\mathfrak{p}$.  We consider the elliptic curves to be defined over the p-adic field $k$. That is, $E_1$ and $E_2$ will mean the base change of the elliptic curves to the p-adic field $k$. We denote the Kummer surface associated to $A$ by $X$. For a choice of $P \in E_1[2], Q \in E_2[2]$ we denote the corresponding Enriques surface of Lieberman type by $S_{P,Q}$. Let $f_{P,Q}: A \dashrightarrow  S$ denote the rational map from the abelian surface to the Enriques surface (we will denote $f_{P,Q}$ by $f$ where it will not cause confusion). We will also abuse notation and write $f_*: A_0(A) \to A_0(S)$ even though $f$ does not define a proper map from $A$ to $S$ since $\tilde{A} \to S$ is proper and $\tilde{A} \to A$ induces an isomorphism on Chow groups.
    \end{set}

    \subsection{The Case of Good Reduction}\label{sec: GoodRed}
        Let us assume that both elliptic curves have good reduction. That is, we assume $E_1$ and $E_2$ have smooth proper models over $\mathcal{O}_k$. In this case, by \cite[Corollary 3.5.1]{RaskindSpiess}, $T(A)$ is 2-divisible. Therefore, $f_*(T(A)) = 0$. However, we can use the work of Saito and Sato to relate the Chow group of $S$ to the Chow group of its reduction over the residue field.
        \begin{lem}
            If the abelian surface $A$ has good reduction, then the surface $S$ has good reduction.
        \end{lem}
        \begin{proof}
            Since the residue field has characteristic not equal to $2$, the assumption that $A$ has good reduction implies the Kummer surface $X$ has good reduction \cite[Lemma 4.2]{Matsumoto}. Also, by the proof of that result, the smooth proper model $\mathcal{X}$ over $\mathcal{O}_k$ is described as follows. Let $\mathcal{A}$ denote the smooth proper model of $A$ over $\mathcal{O}_k$ and let $\tilde{\mathcal{A}}$ denote the blow-up of the abelian scheme at the subscheme $\mathcal{A}[2]$. Then $\mathcal{X}\cong \tilde{\mathcal{A}}/\iota$ where $\iota$ denotes the negation involution on the abelian scheme. Therefore, we can define the involution $\tilde{\sigma}_{P,Q}$ on the scheme $\mathcal{X}$ with the property that its restriction to the generic and special fibers is the usual involution on the Kummer surfaces over the p-adic field and residue field respectively. Therefore, $\mathcal{X}/\langle \tilde{\sigma}_{P,Q} \rangle$ is a proper flat scheme over $\mathcal{O}_k$. The generic fiber is $S$ by definition. And the the special fiber is the Enriques quotient of the Kummer surface $X_{\F_{\mathfrak{p}}}$ by the action $\sigma_{P,Q}$ over $\F_p$. Therefore, $\mathcal{X}/\langle \tilde{\sigma}_{P,Q} \rangle$ is a smooth proper model with generic fiber $S$, so $S$ has good reduction. 
        \end{proof}
        Then by the Lemma, \cite[Théorème 0.1]{Colliot-TheleneChowSuveyPadics} implies that the specialization map
        \[
            A_0(S) \to A_0(S_{\F_{\mathfrak{p}}})
        \]
        is surjective and the kernel is 2-divisible. Since $A_0(S)$ is 16-torsion by Theorem \ref{thm: ChUpperBound}, this implies that the above map is an isomorphism. By \cite[Proposition 9]{KatoSaito}, $A_0(S_{\F_{\mathfrak{p}}})$ is isomorphic to 
        \[
            \Hom(\Hom_{\F_{\mathfrak{p}}}(\mu_n, \text{NS}(S_\mathfrak{p})), \Q/\Z)
        \]
        where $\Hom_{\F_{\mathfrak{p}}}$ denotes the group of homomorphisms of $\F_{\mathfrak{p}}$-group schemes. But since the torsion of the Néron-Severi group of an Enriques surface over any field of characteristic greater than $2$ is $\Z/2\Z$ \cite[Remarks following Theorem 1.3.1]{CossecDolgachevLiedtke}, we have that $\text{NS}(S_\mathfrak{p})_\text{tor}$ is $\Z/2\Z$. Therefore, the above group is isomorphic to $\Z/2\Z$. Thus $A_0(S) \cong \Z/2\Z$ in this case. 

    \subsection{The Case of Split Multiplicative Reduction}
    Given the above, we focus on the case where both elliptic curves have split multiplicative reduction. That is, we assume $E_1$ and $E_2$ have regular proper models over $\mathcal{O}_k$ such that the special fiber has a nodal singularity such that the slopes of the tangent lines are in the residue field $\F_{\mathfrak{p}}$.

    Our goal is to find examples where the pairing 
	\[
	(f)_*T(A)/2 \times \Br(S)[2]/\Br_0(S)[2] \xrightarrow{\langle -, - \rangle_S} \Z/2\Z
	\]
	is nontrivial. This in particular implies that $A_0(S)$ is nontrivial. By Lemma \ref{lem: BMFunctoriality}, this is equivalent to finding examples where 
	\[
	T(A)/2 \times f^*\Br(S)[2]/\Br_0(S)[2] \xrightarrow{\langle -, - \rangle_A} \Z/2\Z
	\]
	is nontrivial. So, let us first consider the pairing 
	\[
	T(A)/2 \times \Br(A)[2]/\Br_0(A)[2] \xrightarrow{\langle -, - \rangle_A}  \Z/2\Z.
	\]
	By Proposition 2.16 of \cite{GazakiKoutsianas}, we have that this pairing factors through 
	\[
	T(A)/2 \times \Br(A)[2]/\Br_1(A)[2] \xrightarrow{\langle -, - \rangle_A}  \Z/2\Z.
	\]
    
    Let $H$ denote the quotient of this group by the right kernel of the pairing $\langle -, - \rangle_A$. In our case, the group $H$ has been computed explicitly and can be described via the p-adic uniformization of the elliptic curves. 
	
	Since $E_1$ and $E_2$ have split multiplicative reduction, there exists $q_1, q_2 \in K^\times$ with $|q_1|_\mathfrak{p} < 1$, $|q_2|_{\mathfrak{p}} < 1$ and an analytic isomorphism $E_i(K) \cong K^\times/q_i^\Z$. Then we can identify $E_i[2]$ with the set $\{\pm \sqrt{q_i^\epsilon}, \epsilon \in {0,1}\}$. So we have a split short exact sequence 
	\[
	   0 \to \{\pm 1\} \to E_i[2] \to \Z/2\Z \to 0,
	\]	
	where the second map sends $\pm \sqrt{q_i^\epsilon} \to \epsilon$. Then by the proof of theorem 4.3 of \cite{Yamazaki}, $H \cong \Hom(\{\pm 1\}, \Z/2\Z)$  where the isomorphism is induced by the quotient 
    \[
        \Br(A)[2]/\Br_1(A)[2] \cong \Hom(E_1[2], E_2[2]) \to \Hom(\{\pm 1\}, \Z/2\Z).
    \]
    So, showing that the pairing 
    \[
	(f)_*T(A)/2 \times \Br(S)[2]/\Br_0(S)[2] \xrightarrow{\langle -, - \rangle_S} \Z/2\Z
	\]
    is nontrivial is equivalent to showing that the composition
	\[
	\Br(S)/\Br_0(S) \to \Br(A)[2]/\Br_1(A)[2] \to \Hom(\{\pm 1\}, \Z/2\Z)
	\]
	is nonzero. We have the following key lemma.
	
	\begin{lem}\label{lem: CondsForNontrivPairing}
		Let $P_0$ correspond to the image of $-1 \in \{\pm \sqrt{q_1^{\epsilon_1}}\}$ and $Q_0$ correspond to the image of $-1 \in \{\pm \sqrt{q_2^{\epsilon_2}}\}$ under the p-adic uniformizations. Then $(f_{P,Q})_*T(A)/2$ pairs nontrivially with $\Br(S_{P,Q})/\Br_0(S_{P,Q})$ if and only if $\mathcal{A}_{\mu, \eta}$ is in the image of $\Br(X)/\Br_1(X)$ where $(\mu, 0) = P \neq P_0$, $(\eta, 0) = Q \neq Q_0$.
	\end{lem}
	\begin{proof}
		($\impliedby$) Let $P = (\mu, 0)$ and $Q = (\eta, 0)$ and let $S$ be the Enriques quotient corresponding to the involution $\sigma_{P, Q}$. By Theorem \ref{thm: CalcBrS}, $\mathcal{A}_{\mu, \eta}$ is in the image of \[\Br(S)[2]/\Br_0(S)[2] \to \Br(A)[2]/\Br_1(A)[2].\] Therefore, by Lemma \ref{lem: SZCanonicalIso} ,$\varphi_{\mu, \eta}$ is in the image of 
        \[
            \Br(S)[2]/\Br_0(S)[2] \to \Br(A)[2]/\Br_1(A)[2] \to \Hom(E_1[2], E_2[2]). 
        \]
        By Lemma \ref{lem: BMFunctoriality} and the remarks following, it suffices to show that $\varphi_{\mu, \eta}$ maps to the nontrivial element of $\Hom(\mu_2, \Z/2\Z)$. By the definition of $\varphi_{\mu, \eta}$, this is equivalent to that $(\mu, 0)$ and $(\eta, 0)$ do not corresponding to -1 under the p-adic uniformization  for their respective curves. This is true by assumption, so we are done. 
		
		($\implies$) The image of 
        \[
            \Br(S)/\Br_0(S) \to \Hom(E_1[2],E_2[2])
        \]is the subgroup generated by $\varphi_{\mu, \eta}$ by Lemma \ref{lem: SZCanonicalIso} and the image of this subgroup under 
        \[
            \Hom(E_1[2],E_2[2]) \to \Hom(\{\pm 1\}, \Z/2\Z)
        \]
        is nontrivial by assumption. This implies that $(\mu, 0) \neq P_0$ and $(\eta, 0) \ne Q_0$. And by Theorem \ref{thm: CalcBrS} we have that $\mathcal{A}_{\mu, \eta}$ must be in the image of $\Br(X)/\Br_1(X)$. 
	\end{proof}

    \begin{rem}\label{rem: -1TwoTorsionPoint}
        For an elliptic curve with minimal Weierstrass model $E: y^2 = x(x-a)(x-b)$ with split multiplicative reduction, under the p-adic uniformization, $-1$ corresponds to the nontrivial 2-torsion point which is nonsingular under the reduction to the residue field. So we will have one of the following cases, $a\equiv 0$, $b\equiv 0$, or $a\equiv b$ in the residue field. These correspond to $-1$ mapping to $(b,0)$, $(a,0)$ or $(0,0)$ respectively under the p-adic uniformization. 
    \end{rem}

    The key lemma will be used to prove the following main result. Recall that we are in Setting \ref{set: PadicSetting} with the further assumption that both elliptic curves have split multiplicative reduction. 
    \begin{thm}\label{thm: MainResult}
        There exists an Enriques quotient of Lieberman type of $X$ such that $f_*(T(A))/2$ pairs nontrivially with $\Br(S)/\Br_0(S)$ if and only if 
        \[
            \Br(X)[2]/\Br_1(X)[2] \to \Br(A)[2]/\Br_1(A)[2]
        \]
        is an isomorphism. 
    \end{thm}

    We prove this theorem by looking at two cases: the first is when $-1$ is a square in $k$ and the second is when $-1$ is not a square. 
	
	\begin{lem}\label{lem: -1Sq}
		Suppose that $-1$ is a square in $k$. Then there exists an Enriques quotient which pairs nontrivially if and only if the map
		\[
		\Br(X)[2]/\Br_1(X)[2] \to \Br(A)[2]/\Br_1(A)[2]
		\]
		is an isomorphism.
	\end{lem}
    \begin{proof}
		If the map is an isomorphism, then it is clear there exists an Enriques quotient which pairs nontrivially by the above. 
		
		Let us assume that there exists an Enriques quotient which pairs nontrivially. We can choose models of the elliptic curves such that 
        \[
            E_1: y^2 = x(x-a)(x-b)
        \]
        \[
            E_2: v^2 = u(u-a')(u-b')
        \]
        such that $v_\mathfrak{p}(a), v_\mathfrak{p}(a') >0$ and $b,b'$ are units in $k$. Then by Remark \ref{rem: -1TwoTorsionPoint}, we get that $(b,0)$ and $(b',0)$ are the points corresponding to $-1$ under the p-adic uniformization. To prove the statement, by Lemma \ref{lem:BrClassesDescendingToKum}, we want to show that the following matrix has all square entries in $k$:
		\[
		\begin{pmatrix}
			1 & ab & a'b' & -aa' \\
			ab & 1 & aa' & a'(a'-b') \\
			a'b' & aa' & 1 & a(a-b) \\
			-aa' & a'(a'-b') & a(a-b) & 1
		\end{pmatrix}.
		\]
        By assumption, both $(b,0)$ and $(-b,0)$ correspond to $-1$ under the p-adic uniformizations, so by Lemma \ref{lem:BrClassesDescendingToKum}, the statement of Lemma \ref{lem: CondsForNontrivPairing} translates to one of the rows above has all square entries. In particular, $aa'$ is a square. Moreover, since $E_1$ and $E_2$ have split multiplicative reduction, $-b, -b'$ are squares, so $b,b'$ are both squares. Also, note that $(a-b)$ has valuation $0$ by assumption, so it is a unit. And mod $\mathfrak{p}$ $(a-b) \equiv-b$ which is a square. Therefore, since $p> 2$, $a-b$ is a square in $k$. By inspecting the above matrix, every row contains either $ab$ or $a(a-b)$, so in either case, $a$ must be a square since one of the above rows has all square entries. Also, by the same argument as above, $(a'-b')$ is also a square. Finally, since $aa'$ is a square and $a$ is a square, so is $a'$. So we have found that $-1, a, a', b, b', a-b,$ and $a'-b'$ are all squares, so every entry in the matrix above is a square which completes the proof. 
	\end{proof}

    \begin{lem}
        Suppose that $-1$ is not a square in $k$. Then for all Enriques quotients of Lieberman type $S$, $f_*(A)/2$ pairs trivially with $\Br(S)[2]/\Br_0(S)[2]$. 
    \end{lem}
    This completes the proof of the above main theorem since the map is never an isomorphism as one of $aa'$ and $-aa'$ are not squares.
    \begin{proof}
        For $E_1$ and $E_2$ we can choose models 
        \[
            E_1: y^2 = x(x-a)(x-b)
        \]
        \[
            E_2: v^2 = u(u-a')(u-b')
        \]
        where $v_\mathfrak{p}(a), v_\mathfrak{p}(a') >0$ and $b,b'$ are units in $k$. Then by these assumptions, $(b,0)$ and $(b',0)$ correspond to $-1$ under the p-adic uniformization. So by Lemma \ref{lem: CondsForNontrivPairing} we want to show that none of the rows of 
        \[
		\begin{pmatrix}
			1 & ab & a'b' & -aa' \\
			ab & 1 & aa' & a'(a'-b') \\
			a'b' & aa' & 1 & a(a-b) \\
			-aa' & a'(a'-b') & a(a-b) & 1
		\end{pmatrix}.
		\]
        has all square entries. By the assumption that $E_1$ and $E_2$ have split multiplicative reduction, $-b$ and $-b'$ must be squares mod $\mathfrak{p}$ which implies that $b,b'$ are not squares mod $\mathfrak{p}$ and since they are units and the prime is odd, we get that $b,b'$ are not squares in $\F_\mathfrak{p}$.  
        \\
        \underline{Case 1: $aa'$ is not a square} \\
        If $aa'$ is not a square the middle two rows cannot have all square entries, so we are left with the top row and bottom row. The top row will not work, because if $ab$ and $a'b'$ are squares, then $(ab)(a'b') = aa'(bb')$ is a square times a nonsquare, so it is a nonsquare which is a contradiction. And we can rule out the bottom row because if $a(a-b)$ is a square, then since $v_\mathfrak{p}(a-b) = 0$, we get that $v_\mathfrak{p}(a(a-b)) = v_\mathfrak{p}(a)$ is even. Also, if $a = w\pi^k$ where $w$ is a unit and $\pi$ is the uniformizer, then the unit part of $a(a-b)$ is $w(w\pi^k - b)$. So mod $\mathfrak{p}$ we get $w(-b)$ which must be a square. But $-b$ is a square, so $w$ is as well. This implies that $a$ is a square and by symmetry so is $a'$. Therefore, $aa'$ would be a square which is a contradiction. 
        \\
        \underline{case 2: $aa'$ is a square} \\
        In this case the top and bottom row can be ruled out immediately. Let us show that the second row cannot have all square entries and by symmetry the same argument applies to the third row. Assume $ab$ is a square. Then this implies that $a$ is not a square since $b$ is not. Therefore, $a'$ is not a square. But if $a'(a'-b')$ is a square then by the exact same argument as in the previous part, $a'$ is a square, which is a contradiction. This completes the proof.
    \end{proof}
	
	So by \cite{SkorobogatovZarhin}, the main theorem above becomes a very easy condition to check and allows us to find examples of such Enriques surfaces. See the next section for an explicit example. 

    \begin{cor}\label{cor: FiniteExt}
        There exists a finite extension $L/k$ and an Enriques quotient $S$ of $X$ such that $A_0(S_L)$ is nontrivial.
    \end{cor}
    \begin{proof}
        By Lemma \ref{lem:BrClassesDescendingToKum}, there exists a finite extension $L/k$ such that after basechanging $X$ and $A$ to $L$, the map on transcendental Brauer groups becomes an isomorphism. Then by Theorem \ref{thm: MainResult}, there exists an Enriques quotient $S_L$ with the property that the induced map $A_0(A_L)\to A_0(S_L)$ is nontrivial, so $A_0(S_L)$ is nontrivial. 
    \end{proof}
	
	\section{Example and Application to Weak Approximation}\label{sec: Example}
	In this section, we give an explicit example of an Enriques surface $S$ such that $A_0(S) \neq 0$. This example is over the p-adic field $\mathbb{Q}_{13}$.
	
	Consider the elliptic curves
	\[
	E_1: y^2 + xy + y = x^3 + x^2 - 565x - 5353 \text{ (LMFDB 390.f3)} 
	\]
	\[
	E_2: y^2 + xy= x^3 -474x +2619 \text{ (LMFDB 429.b3)}.
	\]
    We can write these elliptic curves as 
    \[
        E_1: y^2 = x(x-9)(x-169)
    \]
    \[
        E_2: v^2 = u(u-121)(u-169).
    \]
    Then we can see that both curves have split multiplicative reduction at the prime $p = 13$. Moreover, they are non-isogenous over $\overline{\mathbb{Q}}$ since $E_1$ has split multiplicative reduction at $p=5$ but $E_2$ has good reduction at $p = 5$. We consider these as elliptic curves over $\Q_{13}$. Let $X$ denote the Kummer surface associated with the product of these elliptic curves. Then we claim that $\Br(X)[2]/\Br_1(X)[2] \to \Br(A)[2]/\Br_1(A)[2]$ is an isomorphism. Indeed by Lemma \ref{lem:BrClassesDescendingToKum} it suffices to show that 
    \[ \begin{pmatrix} 
        1 & ab & a'b' & -aa' \\ ab & 1 & aa' & a'(a'-b') \\ a'b' & aa' & 1 & a(a-b) \\ -aa' & a'(a'-b') & a(a-b) & 1 \end{pmatrix}. 
    \]
    has all square entries in $\Q_{13}$. We have $a = 9$, $b=169$, $a' = 121$ and $b'= 169$. So it is easy to compute the above matrix as
    \[
    \begin{pmatrix}
        1 & 39^2 & 143^2 & -33^2 \\
        39^2 & 1 & 33^2 & -3 \cdot 44^2 \\
        143^2 & 33^2 & 1 & -10 \cdot 12^2 \\
        -33^2 & -3 \cdot 44^2 & -10 \cdot 12^2 & 1
    \end{pmatrix}.
    \]
    And since $-3$ and $-10$ are squares in $\Q_{13}$, it is clear that every entry of the matrix is a square, so the above map is an isomorphism. Then by Theorem \ref{thm: MainResult}, there exists an Enriques quotient $S$ of $X$ such that $A_0(S)$ is nontrivial. The Enriques quotient is any of the ones coming from the involutions $\sigma_{P,Q}$ where $P\neq (9,0)$ and $Q \neq (121,0)$.

    \subsection{Remark on Weak Approximation}    
    Our results also provide examples of a transcendental Brauer-Manin obstruction to Weak Approximation on certain Enriques surfaces of Lieberman type $S$. We note that examples of the failure of Weak Approximation have been shown by Harrari and Skorobogatov in \cite{HarariSkorobogatov} and examples using transcendental obstructions have been shown by Creutz and Viray in \cite{CreutzViray}. 
    
    Keep the assumptions of Setting \ref{set: PadicSetting}. Let $v$ be the place corresponding to the prime $\mathfrak{p}$ so that $k = F_v$. Assume that 
    \[
        \Br(X_k)[2]/\Br_1(X_k)[2] \to \Br(A_k)[2]/\Br_1(A_k)[2]
    \]
    is an isomorphism over the local field $k$. Further assume that $\Br(S)/\Br_1(S)$ is nontrivial over the number field $F$. For example, if we take the elliptic curves in the example of the above section over the number field $F = \Q(i)$ with $v$ a place corresponding to a prime over 13 such as $2+3i$, then $k := F_v = \Q_{13}$ and the above conditions are satisfied (by Theorem \ref{thm: CalcBrS} and Lemma \ref{lem:BrClassesDescendingToKum}) for the Enriques surface $S_{P,Q}$ where $P = (0,0) \in E_1(F)$ and $Q = (0,0) \in E_2(F)$. 
    
    For a global Brauer class $\alpha \in \Br(S)$, define the map 
    \[
        \text{ev}_{\alpha}: S(k) \to \Q/\Z
    \]
    by $\text{ev}_{\alpha}(x) := \langle [x], {\alpha}\rangle $. That is taking the image of pairing the zero-cycle associated to the rational point $x$ with the Brauer class $\alpha$. 
    
    \underline{\textbf{Claim}:} There exists a Brauer class $\alpha \in \Br(S)$ such that the map $\text{ev}_\alpha$ is non-constant.
        
    Given the claim we show that Weak Approximation is obstructed for $S$. That is we will show that
    \[
        S(\A_F)^\alpha \subsetneq S(\A_F)
    \]
    where $S(\A_F)^\alpha$ is the set of adelic points which map to 0 after pairing with $\alpha$. This is sufficient as $S(F) \subset S(\A_F)^\alpha$ and the latter set is closed. 
    
    Let $x \in S(F)$ and take $(x_w)_w \in S(\A_F)$ to be the corresponding adelic point. Since $\text{ev}_\alpha$ is nonconstant on $S(k)$ by the Claim, there exists $s_v \in S(k)$ such that $\langle s_v, \alpha\rangle \neq \langle x_v, \alpha\rangle$. Now define the element $(t_w)_w$ by $t_v = s_v$ and $t_w = x_w$ for all $w\neq v$. Then $(t_w)_w \in S(\A_F)$ but by construction it is not an element of $S(\A_F)^\alpha$. Therefore, $\overline{S(F)} \subsetneq S(\A_F)$. Thus it suffices to prove the claim

    \underline{\textbf{Proof of Claim:}}
    Suppose for contradiction that $\text{ev}_\alpha$ is constant for all $\alpha$ in $\Br(S)$. Then the map 
    \[
        \langle -, \alpha \rangle : A_0(S_k)/2 \to \Q/\Z
    \]
    would be trivial for any element of the form $\sum_i [x_i] - [y_i]$ where $x_i, y_i \in S_k(k)$ are rational points of $S_k$. By Remark 4.2.5 of \cite{RaskindSpiess}, $T(A_k)/2$ is generated by elements of the form 
    \[
        (P, Q) - (P,O_2) - (O_1, Q) + (O_1, O_2)
    \]
    for $P \in E_1(k)$ and $Q \in E_2(k)$. Therefore, the map 
    \[
        \langle -, \alpha\rangle : f_*(T(A_k))/2 \to \Q/\Z
    \]
    is trivial by the above assumption since every cycle in $f_*(T(A_k))/2$ can be written as a sum of the difference of rational points. But by Theorem \ref{thm: MainResult}, there exists $\beta \in \Br(S_k)[2]/\Br_0(S_k)[2]$ such that the map 
    \[
        \langle -, \beta \rangle : f_*(T(A_k))/2 \to \Q/\Z
    \]
    is nontrivial. So to arrive at a contradiction it only suffices to show that there exists $\alpha \in \Br(S)$ which maps to a $\beta \in \Br(S_k)[2]/\Br_0(S_k)[2]$ which pairs nontrivially with $f_*(T(A_k))/2$. But any $\beta$ which is nontrivial in $\Br(S_k)/\Br_1(S_k)$ has this property and $\Br(S_k)/\Br_1(S_k) \cong \Br(S)/\Br_1(S) \cong \Z/2\Z$. So any $\alpha \in \Br(S)$ which is nontrivial in $\Br(S)/\Br_1(S)$ has the desired property. This completes the proof of the claim. 
    
	\bibliographystyle{alpha}
	\bibliography{biblio}
	
\end{document}